\documentclass[11pt]{article}
\usepackage{amsthm,amsmath,amssymb,pgf,tikz,pgfpict2e,soul}
\usepackage{graphicx,enumitem}
\usepackage{algpseudocode}
\usepackage{algorithm}
\theoremstyle{plain}
\newtheorem{thm}{Theorem}[section]
\newtheorem{cor}[thm]{Corollary}

\newtheorem{obs}[thm]{Observation}
\newtheorem{lemma}[thm]{Lemma}
\newtheorem{prop}[thm]{Proposition}
\newtheorem{defn}[thm]{Definition}

\newtheorem{exa}[thm]{Example}
\newtheorem{ques}{Question}[section]

\newcommand{\y}{{\color{yellow!60!red}Y}}
\newcommand{\g}{{\color{green!60!black}G}}

\newcommand{\sbar}{\overline{S}}
\DeclareMathOperator{\bnd}{bnd}

\usepackage{authblk}

\title{Discrete-time treatment number of binary trees}

\author[1]{K.~L.~Collins}
\author[2]{M.-E.~Messinger}
\author[3]{A.~N.~Trenk}
\affil[1]{Wesleyan University, CT, USA}
\affil[2]{Mt.~Allison University, NB, CA} 
\affil[3]{Wellesley College, MA, USA}

\date{\today}

\begin{document}

\maketitle

\begin{abstract}

The discrete-time treatment number  of a graph $H$, denoted by $\tau(H)$,  was introduced in \cite{FirstTreatPaper} and arises from a deterministic process in which each vertex is assigned a color at each time-step.   The  pathwidth upper bound $\tau(H)\leq \lceil\frac{1+pw(H)}{2}\rceil$, is shown in \cite{FirstTreatPaper}, where  $pw(H)$ denotes the pathwidth of  graph $H$.  Equality holds when $H$ is the complete binary tree of depth $d$ (denoted by $BT(d)$) and  $1 \le d \le 6$. In this paper, we characterize the sizes of all subsets of vertices of $BT(d)$ 
whose boundary has $3$ or fewer vertices and use this result to prove that $\tau(BT(d))= 3$ for $8\leq d\leq 10$; in these cases, equality also holds in the pathwidth upper bound. By the hereditary property of the treatment number, all 
larger complete binary trees have treatment number at least $3$. In contrast, we  provide an explicit construction to show that $\tau(BT(7))=2$, while  the pathwidth upper bound only shows $\tau(BT(7))\le 3.$
We construct an infinite family of graphs, each with a cut-vertex, whose treatment number depends on the number of components when the cut-vertex is removed. We use a combination of pathwidth and vertex cuts to prove another upper bound on the treatment number and use this to  construct an infinite family of graphs whose boundary size is limited, but whose treatment number is unlimited. 

\end{abstract}

\section{Introduction} 

The treatment model we consider was introduced in~\cite{FirstTreatPaper}. It is a deterministic discrete-time process on a graph in which at every time-step, each vertex is in one of three states:  green, yellow, or red. At the start, each vertex is red. This definition is related to, but different from the Firefighter Problem~\cite{FM} where a contagion occurs in a limited area and spreads.  Recent papers on the Firefighter Problem include~\cite{Burgess, HE} and both of these models fall within a larger class of graph searching problems.

In \cite{FirstTreatPaper} we provide several interpretations for our model, including the following. Consider a classroom of children where each child is in one of three states: engaged (green), losing focus (yellow), or distracted (red). At the start of the lesson, each child is in the red state, and the goal is for all children to  move to the green state. At each time-step, the teacher ``treats" a subset of the children by meeting with them, and this causes them to enter the green state. When a child in the green state has a distracted (red) neighbor, and is not directly re-engaged by the teacher, then the child loses focus and moves into the yellow state. If a child in the yellow state the child does not meet the teacher at the next time-step, then the child becomes distracted and  enters the red state again. Our goal is to transform every child to the green state, i.e., engaged,  at the end of the process.

 A more general model is introduced in~\cite{FirstTreatPaper}: when a vertex enters the green or yellow state, it remains in that state (regardless of the state of its neighbors) for $r$ or $s$ time-steps, respectively.  In this paper, we consider $r=s=1$.
 The \emph{treatment number} of a graph $H$, denoted by $\tau(H)$, is the minimum $k$ for which it is possible to clear $H$ of all red and yellow vertices while treating at most $k$ vertices per time-step.  We make this more precise in Definition~\ref{treat-def} after we define  treatment protocols.

We state some notation and definitions that are essential for introducing the key concepts   in our paper.  If $H$ is a graph and $S \subseteq V(H)$ then we denote by $\sbar$ the  set $V(H) - S$.  For any vertex $v \in V(H)$, we denote its neighborhood set by $N(v)$ and for any $S \subseteq V(H)$ we denote by $N(S)$ the set $\{N(v): v \in S\}$. 
The \emph{boundary of $S$}, denoted by $\bnd(S)$, is the set  $ N(S) \cap \sbar$. Thus, vertices in $\bnd(S)$ are not in set $S$, but are external neighbors of $S$.
For the tree $T$ in Figure~\ref{fig-depth-2-tree} and the set $S = \{v_1,v_2\}$ we have $\bnd(S) = \{v_3,v_4,v_5\}$.

In  Section~\ref{sec-preliminaries} 
we provide definitions and examples as well as background information about the pathwidth of  complete binary trees.
There are three results  we use from \cite{FirstTreatPaper} that provide bounds on $\tau(H)$.  The first is an upper bound on $\tau(H)$ in terms of pathwidth 
(Theorem~\ref{pathwidth-bnd-thm}), 
which we apply 
to bound the treatment number of complete binary trees.  The second is a more complex result, stated in Corollary~\ref{lower-bnd}, 
whose contrapositive  
gives a lower bound on $\tau(H)$ based on the size of boundary sets.  The third is  Proposition~\ref{subgraph}, which shows that the treatment number of any subgraph of $H$ is at most $\tau(H)$.  
It is shown in \cite{FirstTreatPaper} 
that there is a subdivision of any tree  that requires at most two vertices to be treated per time-step, and, in contrast, an open question of \cite{FirstTreatPaper} is finding a smallest tree $T$ with $\tau(T) > 2$. We answer this question for complete  binary trees in Sections~\ref{sec:binarytree}, \ref{BT-8-sec}, and \ref{sec:depth-seven-tree}.

 Indeed, much of this paper focuses on complete binary trees of depth $d$.  We study these graphs in part because their pathwidths are known and in part because it is surprising that the treatment number for a  complete binary tree is two 
 for $d=3$ and remains  at two even when $d=7$.  However, three vertices must be treated at some time-step  when $d=8$ and to prove this we first characterize all boundary sets of size $1$, $2$, or $3$ for complete binary trees.  In Section~\ref{sec:K4s}
 we prove an additional upper bound  based on a vertex cut and demonstrate one of the limitations of our lower bound theorem by constructing a family of graphs with arbitrarily large treatment number but for which the lower bound theorem does not give useful information.

For any graph theory terms not defined here, please consult a standard reference, such as~\cite{West}.

\section{Preliminaries} 
\label{sec-preliminaries}

In this section we introduce fundamental  definitions, illustrate them with an example, and use pathwidth to provide an upper bound on the treatment number of the complete binary tree of depth $d$.

\subsection{Treatment protocols}

As discussed in the introduction, at each time-step we use colors to distinguish between vertices in different states:  green for engaged,
yellow for losing focus, 
and red for distracted.  
We write $G_t$ for the set of green vertices, $Y_t$ for the set of yellow vertices, and $R_t$ for the set of red vertices at time-step $t$.  Initially (i.e., at time-step $t=0$) all vertices are red, 
so $G_0 = Y_0 = \emptyset$.  For $t\ge 1$, the set of vertices treated at time-step $t$ is denoted by $A_t$ and these newly treated vertices are in $G_t$.  This is the only way a vertex can become green.  The transitions of a vertex from green to yellow or yellow to red are given precisely in Algorithm~\ref{alg-treatment} below and described in words subsequently.

\begin{defn} \rm A \emph{protocol} $J$ for a graph $H$ is a sequence $(A_1, A_2, \ldots, A_N)$ of treatment sets, where $A_i \subseteq V(H)$ for $1 \le i \le N$ and $A_i$ is the set of vertices treated during time-step $i$.  The \emph{width} of $J$, denoted $width(J)$, is the maximum value of $|A_i|$ for $1 \le i \le N$.  
\end{defn}
 
\begin{algorithm}[htbp] \small 
\caption{treatment}\label{alg-treatment}
\begin{algorithmic}

\State {\bf Input: } A graph $H$ and a protocol $J= (A_1, A_2, \ldots, A_N)$ for $H$.

\State {\bf Output:} Reports  success if all vertices are green at time-step $N$ and failure otherwise.

\State
{\bf Initialize:}  $R_0 = V(H)$,   $G_0 = Y_0 = \emptyset$ (All vertices are initially red)

\For  {$t := 1 $ to $ N$}

\State $ R_t := (R_{t-1} \cup Y_{t-1})  - A_t$.

\State   $Y_t := \{v \in G_{t-1} | \, \exists x \in N(v) $ with $ x \in R_t \} - A_t$ 
 
\State  $G_t := A_t \cup \{v \in G_{t-1} | \, \forall x \in N(v)$ we have  $ x \not\in R_t \}$
 
\EndFor

\If{$G_N = V(H)$,} 
\State {report  ``success."  } 
\Else { report  ``failure."}
 \EndIf
\end{algorithmic}
\end{algorithm}

The three steps of the  loop in the algorithm can be expressed in terms of membership in the three states as follows.  Vertices that are    red or yellow  at time-step $t-1$ will be or become  red at time-step $t$ unless they are treated.  The vertices that are  yellow at time-step $t$ are those that were green at time-step $t-1$, not treated at time-step $t$, and have a neighbor that is red at time-step $t$.  The green vertices at time-step $t$ are those that are newly treated, together with those that were green in the previous time-step and have no neighbors that are red at time-step $t$.\\

  \begin{defn} \rm A protocol $J $ {\it clears} graph $H$, if Algorithm~\ref{alg-treatment} reports ``success" when the input is  graph $H$ and protocol $J$.  
   \label{treated-sets}
 \end{defn}

\begin{defn} \rm 
 The \emph{treatment number} of a graph $H$, denoted by $\tau(H)$, is the smallest width of a protocol that clears $H$.
 \label{treat-def}
 \end{defn}

\begin{figure}[hb]
\begin{center}
\begin{tikzpicture}[scale=.85]
\draw[thick] (-.3,0)--(.5,1)--(1,0); 
\draw[thick] (2,0)--(2.5,1)--(3.3,0); 
\draw[thick] (.5,1)--(1.5,1.8)--(2.5,1); 

\draw[thick] (5.7,0)--(6.5,1)--(7,0); 
\draw[thick] (8,0)--(8.5,1)--(9.3,0); 
\draw[thick] (6.5,1)--(7.5,1.8)--(8.5,1); 

\filldraw[black]
(1.5,1.8) circle [radius=3pt]
(2.5,1) circle [radius=3pt]
(.5,1) circle [radius=3pt]
(3.3,0) circle [radius=3pt]
(2,0) circle [radius=3pt]
(1,0)circle [radius=3pt]
(-.3,0)circle [radius=3pt]
;

\filldraw[black]
(7.5,1.8) circle [radius=3pt]
(8.5,1) circle [radius=3pt]
(6.5,1) circle [radius=3pt]
(9.3,0) circle [radius=3pt]
(8,0) circle [radius=3pt]
(7,0)circle [radius=3pt]
(5.7,0)circle [radius=3pt]
;

\node(ell2) at (1.97,2.2) {$v_1$ {\color{green!65!black} $[1,3]$}};
\node(ell1) at (-.2,1.2) {{\color{green!65!black} $[1]$} $v_2$};
\node(ell1) at (3.1,1.2) {$v_5$ {\color{green!65!black} $[3]$}};

\node(ell0) at (-.6,-.4) {{\color{green!65!black} $[2]$} $v_3$};
\node(ell0) at (.8,-.4) {{\color{green!65!black} $[2]$} $v_4$};
\node(ell0) at (2.3,-.4) {$v_6$ {\color{green!65!black} $[4]$}};
\node(ell0) at (3.65,-.4) {$v_7$ {\color{green!65!black} $[4]$}};

\node(ell2) at (7.7,2.2) {$v_1$ {\color{green!65!black} $[6]$}};
\node(ell1) at (5.4,1.2) {{\color{green!65!black} $[1,3,5]$} $v_2$};
\node(ell1) at (9.35,1.2) {$v_5$ {\color{green!65!black} $[7,9]$}};

\node(ell0) at (5.4,-.4) {{\color{green!65!black} $[2]$} $v_3$};
\node(ell0) at (6.8,-.4) {{\color{green!65!black} $[4]$} $v_4$};
\node(ell0) at (8.3,-.4) {$v_6$ {\color{green!65!black} $[8]$}};
\node(ell0) at (9.7,-.4) {$v_7$ {\color{green!65!black} $[10]$}};

\end{tikzpicture}
\end{center}

\caption{  
Two protocols  that clear  $BT(2)$,   with  the times that each vertex is treated in brackets (see  Example~\ref{depth-2-tree-example}).}

\label{fig-depth-2-tree}
\end{figure}
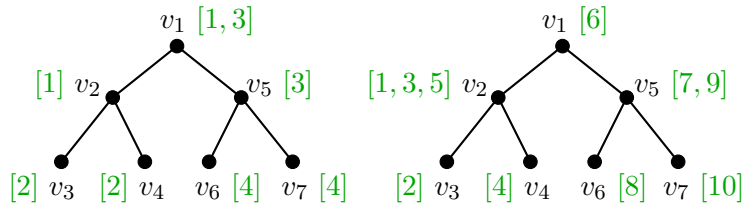 

Binary trees play a fundamental role in this paper.  
The \emph{complete binary tree of depth $d$} is the rooted tree in which every non-leaf vertex has exactly two children and every leaf is exactly distance $d$ from the root.  We  write $BT(d)$ to denote the complete binary tree of depth $d$.  Figure~\ref{fig-depth-2-tree}  shows the tree $BT(2)$ with  vertex set $ \{v_1,v_2,v_3,v_4,v_5,v_6,v_7\}$.

\begin{exa}\label{depth-2-tree-example} {\rm
 The treatment sets
$A_1 = \{v_1, v_2\}$, $A_2 = \{v_3, v_4\}$, $A_3 = \{v_1, v_5\}$, $A_4 = \{v_6, v_7\}$, illustrated on the left of  Figure~\ref{fig-depth-2-tree}, form a width $2$ protocol $ J=(A_1,A_2,A_3,A_4)$ for the graph.   Table~\ref{Alg-implement-table} shows the values of  of $R_t$, $Y_t$, and $G_t$ when  protocol $J$  is applied to $BT(2)$.
 Since  all vertices are  green at time-step 4,   we conclude that $J$ clears tree $BT(2)$, and consequently, its treatment number is at most $2$.  In fact,  $
BT(2)$ is a caterpillar, so its treatment number is $1$, as shown in \cite{FirstTreatPaper}. The protocol shown on the right in Figure~\ref{fig-depth-2-tree} has width 1 and 
 there are additional protocols of width $1$ that clear $BT(2)$ in  even fewer time-steps. }

\end{exa}

\begin{table}[ht] 
\small 
\centering 
\begin{tabular}{| l | l | l | c | l |}\hline
$t$ & $A_t$ &  $R_t$ & $Y_t$  & $G_t$\\  \hline 
$0$ &  & $\{v_1,v_2, v_3, v_4, v_5, v_6, v_7\}$ & $\emptyset$ &  $\emptyset$ \\ \hline
$1$ & $\{v_1, v_2\}$  & $\{ v_3, v_4, v_5, v_6, v_7\}$ & $\emptyset$ &  $ \{v_1, v_2\}$ \\ \hline
$2$ & $\{v_3, v_4\}$ & $\{  v_5, v_6, v_7\}$ & $\{v_1\}$  &  $ \{ v_2,v_3,v_4\}$ \\ \hline
$3$ & $\{v_1, v_5\}$ & $\{  v_6, v_7\}$ & $\emptyset $  &  $ \{ v_1, v_2,v_3,v_4, v_5\}$ \\ \hline
$4$ & $\{v_6, v_7\}$ & $\emptyset$ &$\emptyset$  &  $ \{ v_1, v_2,v_3,v_4, v_5, v_6, v_7\}$ \\ \hline
\end{tabular}
\caption{   The values of $R_t$, $Y_t$, and $G_t$  when  protocol $(A_1, A_2, A_3, A_4)$ of Example~\ref{depth-2-tree-example}
is applied to $BT(2)$.}
\label{Alg-implement-table}
\end{table}

\subsection{Pathwidth of binary trees}
We begin by stating the definitions of a path decomposition and the pathwidth of a graph.

\begin{defn} \rm
A \emph{path decomposition} of a graph $H$ is a sequence $(B_1,\dots,B_m)$ where each $B_i$ is a subset of $V(H)$ such that (i) if $xy \in E(H)$ then there exists an $i$, for which $x,y \in B_i$ and (ii)
for all $w \in V(H)$, the set of $B_i$ containing $w$ forms a consecutive subsequence; that is, if $w \in B_i$ and $w \in B_k$ and $i \leq j \leq k$ then $w \in B_j$. The subsets $B_1,B_2,\dots,B_m$ are often referred to as ``bags'' of vertices.

The \emph{width} of a path decomposition $(B_1,\dots,B_m)$ is $\max_{1 \leq i \leq m} |B_i|-1$ and the \emph{pathwidth} of a graph $H$, denoted $pw(H)$, is the minimum width taken over all path decompositions of $H$.
\label{def:pathwidth}
\end{defn}

The tree $BT(2)$ shown in Figure~\ref{fig-depth-2-tree}  has pathwidth $1$, and this is achieved using the path decomposition $(B_1, B_2, \ldots, B_6)$
where $B_1 = \{v_2,v_3\}$, 
$B_2 = \{v_2,v_4\}$, 
$B_3 = \{v_1,v_2\}$, 
$B_4 = \{v_1,v_5\}$, 
$B_5 = \{v_5,v_6\}$, 
$B_6 = \{v_5,v_7\}$. 
An upper bound for the treatment number of a graph in terms of its pathwidth is given in 
\cite[Theorem~3.4]{FirstTreatPaper} and we restate it here in the case $r = s = 1$.

\begin{thm}[\cite{FirstTreatPaper}] \label{thm3.4}  If $H$ is a graph, then $\tau(H) \leq \left \lceil \frac{1+pw(H)}{2} \right \rceil$.
\label{pathwidth-bnd-thm}
\end{thm}

Stated in~\cite{Bo98} and attributed to~\cite{Scheffler}, the pathwidth of $BT(d)$ is $\left \lceil \frac{d}{2} \right \rceil$.  Combining these two results yields the next corollary.

\begin{cor}\label{cor:perfbin}  The complete binary tree $BT(d)$ satisfies $\tau(BT(d)) \leq \Big\lceil \frac{1+\left \lceil \frac{d}{2}\right \rceil}{2}\Big\rceil$.
\end{cor} 

We conclude this section by applying Corollary~\ref{cor:perfbin} for values of $d$ with   $d\le 10$ and  previewing results of Sections~\ref{sec:binarytree} and \ref{sec:depth-seven-tree}
that enable us to find  the exact value of $BT(d)$ for all $d$ with $d \le 10$.   
Example~\ref{depth-2-tree-example}  shows that $\tau(BT(2))= 1$.  
In~\cite[Theorem~4.3]{FirstTreatPaper}, it is shown that among graphs  with at least one edge,  the ones with treatment number $1$ are precisely the caterpillars, thus $\tau(BT(d)) \ge 2$ for $d\ge 3$.  Coupling this with Corollary~\ref{cor:perfbin}, we find $\tau(BT(d))=2$  for $d \in \{3,4,5,6\}$ and $\tau(BT(d)) \in \{2,3\}$  for $d \in \{7,8,9,10\}$.

In  Section~\ref{BT-8-sec} we prove that any protocol that clears $BT(8)$ requires three vertices to be treated at  some time-step, and this allows us to show
that $\tau(BT(d)) = 3$ for $d \in \{8,9,10\}$. 
Surprisingly, $\tau(BT(7)) = 2$, which we show in 
Section~\ref{sec:depth-seven-tree} by providing an explicit width $2$ protocol. 

\section{Subsets of $BT(d)$ with boundary size at most 3}
\label{sec:binarytree}

 The upper bound  $\tau(BT(8)) \le 3$ follows from Corollary~\ref{cor:perfbin}, but our proof that there is no width $2$ protocol that clears $BT(8)$ requires more tools.  One of these tools is of independent interest:  characterizing all possible values of $|S|$ when $S$ is a set of vertices in $BT(d)$ with $|\bnd(S)| \le 3$, and describing these sets $S$.  This is presented in the statement and proof of Theorem~\ref{binary-tree-thm}.

We begin by noting that if a connected graph $H$ contains both vertices of  a set $S$ and vertices of $\sbar$ then $H$ contains a vertex in $\bnd(S)$.  We record this  idea in the following where $H$ is  a subtree of $BT(d)$ rooted at a particular vertex.

\begin{obs}
 Suppose  $S$ is a set of vertices in $BT(d)$,  $v$ is a vertex of $BT(d)$, and $T_v$ is the subtree of $BT(d)$ rooted at $v$.   If no vertex of $T_v$ is  in $\bnd(S)$ then either all vertices of $T_v$ are in $S$ or all are in $\sbar$. 
 \label{rooted-tree-d}
\end{obs}
In Theorem~\ref{binary-tree-thm} we characterize the possible values of $|S|$ where $S$ is a set of vertices in graph $H$ with $|\bnd(S)| \le 3$.  We prove Theorem~\ref{binary-tree-thm} by classifying vertices in $\bnd(S)$ 
into  categories based on Definition~\ref{top-bottom-def} below.

\begin{defn} {\rm Let $r$ be the root of $BT(d)$  and let $S \subseteq V(BT(d))$.  

\noindent
(i) \ If  $r \in \sbar$ then a vertex $v$ is called a \emph{top vertex} if the path from $v$ to $r$ contains no vertices in $S$ and one or both children of $v$ are in $S$.

\noindent
(ii) \ If $r \in S$ then a vertex $w$ is called a \emph{bottom vertex} if $w\not\in S$ but all the interior vertices on the path from $w$ to $r$ are in $S$.   
}
\label{top-bottom-def}
\end{defn}

Figure~\ref{2trees-fig} illustrates Definition~\ref{top-bottom-def}.  In each of these rooted trees, the vertices of a  set $S$ are solid, the vertices in $\sbar$ are hollow, and the vertices in $\bnd(S)$ are boxed. In the graph on the left, $v_1, v_2$, and $v_3$ are top vertices while in the graph on the right, only $v_1$ and $v_3$ are top vertices.  There are no bottom vertices in Figure~\ref{2trees-fig} since $r \in \sbar$  in both graphs. The next lemma provides some facts that will be used in the proof of Theorem~\ref{binary-tree-thm}.

\begin{figure}[htb] 
\begin{center}
\begin{tikzpicture}[scale=.525,squarednode/.style={rectangle, draw=orange!80!black, fill=red!10!white, very thick, minimum size=4pt}]

\draw[thick](3.5,2.7)-- (7.5,3.6)--(11.5,2.7);

\draw[thick] (9.5,1.8)--(11.5,2.7)--(13.5,1.8);

\draw[thick] (12,0)--(12.5,1)--(13,0); 
\draw[thick] (14,0)--(14.5,1)--(15,0); 
\draw[thick] (12.5,1)--(13.5,1.8)--(14.5,1); 

\draw[thick] (8,0)--(8.5,1)--(9,0); 
\draw[thick] (10,0)--(10.5,1)--(11,0); 
\draw[thick] (8.5,1)--(9.5,1.8)--(10.5,1); 

\draw[thick] (1.5,1.8)--(3.5,2.7)--(5.5,1.8);

\draw[thick] (4,0)--(4.5,1)--(5,0); 
\draw[thick] (6,0)--(6.5,1)--(7,0); 
\draw[thick] (4.5,1)--(5.5,1.8)--(6.5,1); 

\draw[thick] (0,0)--(.5,1)--(1,0); 
\draw[thick] (2,0)--(2.5,1)--(3,0); 
\draw[thick] (.5,1)--(1.5,1.8)--(2.5,1); 

\node(r) at (7.5,4) {$r$};
\node(v) at (3.4,3.4){$v_1$};
\node(v) at (1.5,2.45){$v_2$};
\node(v) at (13.5,2.45){$v_3$};
\node[squarednode] at (3.5,2.7) {};
\node[squarednode] at (1.5,1.8) {};
\node[squarednode] at (13.5,1.8) {};

\filldraw[color=black, fill=white]
(7.5,3.6) circle [radius=4pt]
(11.5,2.7) circle [radius=4pt];

\draw[black]
(13.5,1.8) circle [radius=4pt];

\filldraw[blue!65!black]
(14.5,1) circle [radius=4pt]
(12.5,1) circle [radius=4pt]
(15,0) circle [radius=4pt]
(14,0) circle [radius=4pt]
(13,0)circle [radius=4pt]
(12,0)circle [radius=4pt];
\filldraw[color=black, fill=white]
(9.5,1.8) circle [radius=4pt]
(10.5,1) circle [radius=4pt]
(8.5,1) circle [radius=4pt]
(11,0) circle [radius=4pt]
(10,0) circle [radius=4pt]
(9,0)circle [radius=4pt]
(8,0)circle [radius=4pt];
\draw[black]
(3.5,2.7) circle [radius=4pt];

\filldraw[blue!65!black]
(5.5,1.8) circle [radius=4pt]
(6.5,1) circle [radius=4pt]
(4.5,1) circle [radius=4pt]
(7,0) circle [radius=4pt]
(6,0) circle [radius=4pt]
(5,0)circle [radius=4pt]
(4,0)circle [radius=4pt];
\draw[black]
(1.5,1.8) circle [radius=4pt];
\filldraw[color=black, fill=white]
(2.5,1) circle [radius=4pt]
(3,0) circle [radius=4pt]
(2,0) circle [radius=4pt];
\filldraw[blue!65!black]
(.5,1) circle [radius=4pt]
(1,0)circle [radius=4pt]
(0,0)circle [radius=4pt]
;


\draw[thick](19.5,2.7)-- (23.5,3.6)--(27.5,2.7);

\draw[thick] (25.5,1.8)--(27.5,2.7)--(29.5,1.8);

\draw[thick] (28,0)--(28.5,1)--(29,0); 
\draw[thick] (30,0)--(30.5,1)--(31,0); 
\draw[thick] (28.5,1)--(29.5,1.8)--(30.5,1); 

\draw[thick] (24,0)--(24.5,1)--(25,0); 
\draw[thick] (26,0)--(26.5,1)--(27,0); 
\draw[thick] (24.5,1)--(25.5,1.8)--(26.5,1); 

\draw[thick] (17.5,1.8)--(19.5,2.7)--(21.5,1.8);

\draw[thick] (20,0)--(20.5,1)--(21,0); 
\draw[thick] (22,0)--(22.5,1)--(23,0); 
\draw[thick] (20.5,1)--(21.5,1.8)--(22.5,1); 

\draw[thick] (16,0)--(16.5,1)--(17,0); 
\draw[thick] (18,0)--(18.5,1)--(19,0); 
\draw[thick] (16.5,1)--(17.5,1.8)--(18.5,1); 

\node(r) at (23.5,4) {$r$};
\node(v) at (19.4,3.4){$v_1$};
\node(v) at (16.3,1.65){$v_2$};
\node(v) at (29.5,2.45){$v_3$};
\node[squarednode] at (19.5,2.7) {};
\node[squarednode] at (16.5,1) {};
\node[squarednode] at (29.5,1.8) {};

\filldraw[color=black, fill=white]
(23.5,3.6) circle [radius=4pt]
(27.5,2.7) circle [radius=4pt];

\draw[black]
(29.5,1.8) circle [radius=4pt];

\filldraw[blue!65!black]
(30.5,1) circle [radius=4pt]
(28.5,1) circle [radius=4pt]
(31,0) circle [radius=4pt]
(30,0) circle [radius=4pt]
(29,0)circle [radius=4pt]
(28,0)circle [radius=4pt];

\draw[black] (19.5,2.7) circle [radius=4pt];
\filldraw[color=black, fill=white]
(25.5,1.8) circle [radius=4pt]
(26.5,1) circle [radius=4pt]
(24.5,1) circle [radius=4pt]
(27,0) circle [radius=4pt]
(26,0) circle [radius=4pt]
(25,0)circle [radius=4pt]
(24,0)circle [radius=4pt]

(21.5,1.8) circle [radius=4pt]

(22.5,1) circle [radius=4pt]
(20.5,1) circle [radius=4pt]
(23,0) circle [radius=4pt]
(22,0) circle [radius=4pt]
(21,0)circle [radius=4pt]
(20,0)circle [radius=4pt];
\filldraw[blue!65!black]
(17.5,1.8) circle [radius=4pt];
\filldraw[color=black, fill=white]
(18.5,1) circle [radius=4pt]
(19,0) circle [radius=4pt]
(18,0) circle [radius=4pt]
(17,0)circle [radius=4pt]
(16,0)circle [radius=4pt]
;
\draw[black]
(16.5,1) circle [radius=4pt];

\end{tikzpicture}
\end{center}

\caption{Two trees illustrating Cases 1 and 2  of the proof of Theorem~\ref{binary-tree-thm}.  The vertices in set $S$ are solid, and in both trees,  $\bnd(S) = \{v_1,v_2,v_3\}$.}  

\label{2trees-fig} 
\end{figure}
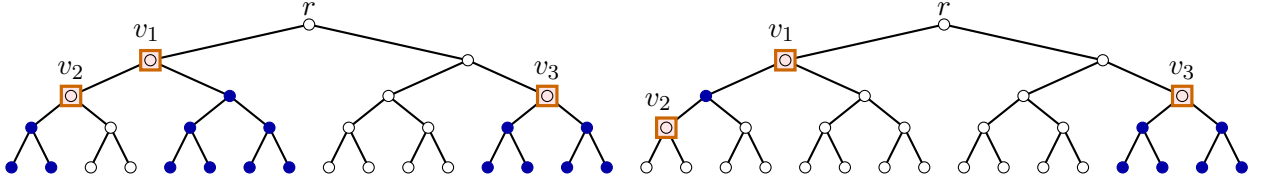

\begin{lemma}
 Let $r$ be the root of $BT(d)$  and   let $S \subseteq V(BT(d))$. 

{\rm (a)} Top vertices and bottom vertices are in $\bnd(S)$. 

{\rm (b)} If $r \in \sbar$ and $v$ is a non-top vertex in $\bnd(S)$ then there exists a top vertex on the path from $v$ to $r$.

{\rm (c)} If $r \in \sbar$ then each vertex $s \in S$ is associated with a unique top vertex, namely the first top vertex on the path from $s$ to $r$.

{\rm (d)} If $r \in S$ and $v$ is a non-bottom vertex in $\bnd(S)$ then there exists a bottom vertex on the path from $v$ to $r$.

\label{tree-lem}
\end{lemma}

\begin{proof}
(a)  Recall that $\bnd(S) = N(S) \cap \sbar$.  First suppose that $v$ is a top vertex and apply Definition~\ref{top-bottom-def}.  The path from $v$ to $r$ consists entirely of vertices in $\sbar$, so $v \in \sbar$.  Furthermore, at least one child of $v$ is in $S$, so $v \in N(S)$.  Thus $v \in  N(S) \cap \sbar$ as desired.

Now suppose $w$ is a bottom vertex.  Again applying Definition~\ref{top-bottom-def} we get $w \in \sbar$ and the neighbor of $w$ on the $wr$-path is in $S$, so $w \in  N(S) \cap \sbar$.

(b)  Since $v$ is not a top vertex, there exists a vertex of $S$ on the $vr$-path.  The parent of the last $S$ vertex on this path will be a top vertex  on the  $vr$-path.

(c) 
 Since $r \not\in S$,  for each $s \in S$, there must be at least one top vertex on the path from $s$ to $r$. The top vertex associated with $s$ is the parent of the last $S$ vertex on the path from $s$ to $r$, and this is also the first top vertex on the path from $s$ to $r$.
 
 (d) Since $v$ is not a bottom vertex, there exists a vertex of $\sbar$ on the $vr$-path.  The last $\sbar$ vertex on this path will be a bottom vertex on the  $vr$-path.
\end{proof}

By characterizing vertex sets in $BT(d)$ that have boundary size at most $3$, we also characterize the  possible sizes of these sets in the next theorem.

\begin{thm}
If $S$ is a set of vertices of $BT(d)$ with $|\bnd(S)| \le 3$, then  $|S|$ satisfies one of the equations in Table~\ref{equns-table}. 
\label{binary-tree-thm}
\end{thm}

\begin{proof}
Let $r$ be the root of $BT(d)$. Throughout this proof we assume that $t_1, t_2, t_3$ are positive integers, each at most $d+1$ and that $b$ is an integer. 
If $\bnd(S) = \emptyset$ then either $S = \emptyset$ or $S = V(BT(d))$.   In the first case, $|S|$ satisfies equation 1 of Table~\ref{equns-table} with $t_1 = 1$ and $b=2$  and in the second case $|S|$ satisfies the same equation with $t_1 = d+1$ and $b=1$. 
  Thus we may assume $\bnd(S) \neq \emptyset.$ 
We know that top and bottom vertices are in $\bnd(S)$ by Lemma~\ref{tree-lem}(a), and
our cases  below depend on how many  additional boundary vertices exist.  In the first three cases the root is   in $\sbar$, so by Lemma~\ref{tree-lem}(b), there are one or more top vertices.   In the second three cases, the root is  in $S$, so  by Lemma~\ref{tree-lem}(d), there are one or more bottom vertices.
\medskip

\noindent {\bf Case 1: } 
$r \in \sbar$ and  all vertices in $\bnd(S)$ are top vertices.  

This case is illustrated by the graph on the left of Figure~\ref{2trees-fig}.
By Lemma~\ref{tree-lem}(c),  each vertex $s \in S$ is associated with a unique top vertex, thus we  can count the vertices of $S$ by counting those associated with each top vertex. 
For each top vertex $v_i$, either one or  both of its children are in $S$ by Definition~\ref{top-bottom-def}. If a child $x_i$ is in $S$ then by the hypothesis of this case, there are no boundary vertices of $S$ in the subtree rooted at $x_i$, so by Observation~\ref{rooted-tree-d}, all vertices in this subtree are in $S$.   If both children of $v_i$ are in $S$ then  both full  subtrees below $v_i$ are in $S$ and there are $2^{t_i}-2$ vertices of $S$ associated with $v_i$.  If one child $x_i$ is in $S$ and the other $y_i$ is not,  then the full subtree rooted at  $x_i$ is in $S$ and this contains $2^{t_i}-1$ vertices. Any vertices at or below $y_i$ that are in $S$ are associated with a  top vertex that is different from $v_i$.  Thus  for each top vertex, we get $2^{t_i}-d_i$ vertices of $S$ associated with it, where $d_i \in \{1, 2\}$.  
There are three possibilities depending on whether $\bnd(S)$ contains 1, 2, or 3 vertices, and the size of $S$ in each of these cases is shown   in Table~\ref{equns-table}, equations 1, 2, and 3.
\medskip

\noindent {\bf Case 2: } $r \in \sbar$  and there is exactly one vertex in $\bnd(S)$ that is not a  top vertex. 

This case is illustrated by the graph on the right of Figure~\ref{fig-depth-2-tree}.
Let $v_2$ be the vertex in $\bnd(S)$ that is not a top vertex.  By the definition of boundary we know $v_2 \in \sbar$ and by Lemma~\ref{tree-lem}(b), there exists a top vertex $v_1$  that is the first top vertex encountered on the path from $v_2$ to $r$.   Since $v_2$ is not a top vertex, the $v_1v_2$ path must contain a vertex of $S$.  Indeed, all internal vertices on this path must be in $S$ or there would be another  non-top boundary vertex between $v_1$ and $v_2$, contradicting the hypothesis of this case.

 In our calculation of $|S|$ in case 1, we counted all vertices in the subtree rooted at $v_2$ as belonging to  $S$.  We now modify this calculation since  in this subtree, either (i) only $v_2 \in \sbar$, (ii) $v_2$ and one full  subtree are in $\sbar$ (and the other full  subtree is in $S$), or (iii) $v_2$ and both full  subtrees are in $\sbar$.  In (i) we subtract $1 = 2^{1} -1$, in (ii) we subtract $2^{t_2}$, and in (iii) we subtract $2^{t_2}-1$ from the possible values of $|S|$ given in equations 1 and 2.  There are two possibilities (cases 2a, 2b) depending on whether $\bnd(S)$ contains 2 or 3 vertices.    In case 2a we have  
 $|S| = 2^{t_1} - 2^{t_2}   - b$ for some $b$ with $0 \le b \le 2$.  In case 2b   we have
$|S| = 2^{t_1} - 2^{t_2} + 2^{t_3} - b$  with $1 \le b \le 4$. These appear in equations 4 and 5, respectively.

\begin{table}[ht]

\[\fbox{\parbox{0.8\textwidth}{
\begin{enumerate}
\small 

\item   $|S| = 2^{t_1}    - b$ for some $b$ with $1 \le b \le 2$
\item   $|S| = 2^{t_1} + 2^{t_2}   - b$ for some $b$ with $2 \le b \le 4$
\item    $|S| = 2^{t_1} + 2^{t_2} + 2^{t_3} - b$ for some $b$ with $3 \le b \le 6$

\item $|S| = 2^{t_1} - 2^{t_2}   - b$ for some $b$ with $0 \le b \le 2$
\item  $|S| = 2^{t_1} - 2^{t_2} + 2^{t_3} - b$ for some $b$ with $1 \le b \le 4$

\item  $|S| = 2^{t_1} - 2^{t_2} - 2^{t_3} - b$ for some $b$ with $-1 \le b \le 2$

\item  $|S| = (2^{d+1}-1) - 2^{t_1}     + b$ for some $b$ with $0 \le b \le 1$
\item    $|S| = (2^{d+1} - 1)  - 2^{t_1} - 2^{t_2}   + b$ for some $b$ with $0 \le b \le 2$
\item    $|S| = (2^{d+1} - 1) - 2^{t_1} - 2^{t_2} - 2^{t_3} + b$ for some $b$ with $0 \le b \le 3$

 \item  $|S| = (2^{d+1} - 1)  - 2^{t_1} + 2^{t_2}   - b$ for some $b$ with $0 \le b \le 2$
  \item  $|S| = (2^{d+1} - 1)  - 2^{t_1} + 2^{t_2} - 2^{t_3}   - b$ for some $b$ with $-1 \le b \le 2$

\item  $|S| = (2^{d+1} -1) - 2^{t_1} + 2^{t_2} + 2^{t_3}  - b$ for some   $  b$ with $1 \le b \le 4$. 
 \end{enumerate} }}\]
 
 \caption{Possible values of $|S|$   when $S$ is a set of vertices in $BT(d)$ and $|\bnd(S)| \le 3$.The values $t_1, t_2, t_3$ are positive integers and $b$ is an integer. }
 
 \label{equns-table}
\end{table}

\medskip 

\noindent {\bf Case 3: } $r \in \sbar$ and there are exactly two vertices in $\bnd(S)$ that are not top vertices. 

The root is in $\sbar$ and there are exactly two vertices in $\bnd(S)$ that are not top vertices. 

In this case, $|\bnd(S)| = 3$ and there is exactly one top vertex.  Let $\bnd(S) = \{v_1,v_2,v_3\}$ where $v_1$ is the top vertex.  By Lemma~\ref{tree-lem}(b), we know there is a top vertex above each of $v_2$ and $v_3$.  Since $v_1$ is the only top vertex, we know $v_1$ is on both the $v_2r$-path and the $v_3r$-path.  Without loss of generality, we may assume that $v_3$ is not on the $v_1v_2$-path.

Recall that in case 2a there is one top vertex ($v_1$), and one non-top boundary vertex ($v_2$), and all internal vertices on the $v_1v_2$-path are in $S$.  We modify the calculation of $|S|$ in case 2a by considering the parent of vertex $v_3$.  If  the parent of $v_3$ is in $S$ then the subtree rooted at $v_3$ was counted as being in $S$ in case 2a, so we \emph{subtract} $2^{t_3} - c_3$ where $c_3 \in \{0,1\}$, as these are the possible number of $\sbar$ vertices in the subtree rooted at $v_3$.  The result is $|S| = 2^{t_1} - 2^{t_2} - 2^{t_3} -b $ with $-1 \le b \le 2$.  This appears in equation 6. 

If  the parent of $v_3$ is in $\sbar$ then the subtree rooted at $v_3$ was counted as being in $\sbar$ in case 2a, so we \emph{add} $2^{t_3} - d_3$ where $d_3 \in \{1,2\}$, as these are the possible number of $S$ vertices in the subtree rooted at $v_3$.  The result is $|S| = 2^{t_1} - 2^{t_2} + 2^{t_3} -b $ with $1 \le b \le 4$.  This appears in equation 5.  

\medskip
\noindent {\bf Case 4: }  $r \in S$ and all vertices in $\bnd(S)$ are bottom vertices. 

By Definition~\ref{top-bottom-def}, no bottom vertex is below another, and by the hypothesis of this case, all boundary vertices are bottom vertices.  Thus there are no boundary vertices below a bottom vertex.  Using Observation~\ref{rooted-tree-d},  for each bottom vertex $w_i$, there are three possibilities.  In the first instance, all vertices below $w_i$ are in $\overline{S}$, so there are $2^{t_i}-1$ vertices of $\overline{S}$ associated with $w_i$ (including $w_i$).  In the second instance,  all vertices below $w_i$ are in $S$, so there is $1 = 2^1-1$ vertex of $\overline{S}$ associated with $w_i$ (namely $w_i$).  In the third possibility,  one child and its descendants  are in  $\overline{S}$ while the other child and its descendants are in $S$, so there are $2^{t_i}$ vertices of  $\overline{S}$ associated with $w_i$.    Thus for each bottom vertex $w_i$ there are $2^{t_i} - c_i$ vertices of $\overline{S}$ associated with $w_i$ where $c_i \in \{0,1\}$.   To find the number of vertices in $S$, we subtract this from $|V(T)| = 2^{d+1} -1$.
There are three possibilities (subcases 4a, 4b, 4c) depending on whether $\bnd(S)$ contains 1, 2, or 3 vertices respectively.  In (4a) we have $|\sbar| = 2^{t_1} - b$ for some $b$ with $0 \le b \le 1$, in (4b)  we have $|\sbar| = 2^{t_1} + 2^{t_2} - b$  with $0 \le b \le 2$, and in (4c)  we have $|\sbar| = 2^{t_1} + 2^{t_2} + 2^{t_3}- b$  with $0 \le b \le 3$.  
The size of $S$ in each of these cases is given in equations 7, 8, and 9.

\medskip 

\noindent {\bf Case 5: }  $r \in S$ and there is exactly one vertex in $\bnd(S)$ that is not a  bottom vertex. 

Again by Definition~\ref{top-bottom-def},  no bottom vertex can be below another bottom vertex, and by Lemma~\ref{tree-lem}(d),  any vertex in $\bnd(S)$ that is not a bottom vertex \emph{must} be below a bottom vertex.  Let $w_2$ be the vertex in $\bnd(S)$ that is not a bottom vertex and let $w_1$ be the bottom vertex on the path from $w_2$ to $r$.

First consider case 5a in which the parent of $w_2$ is in $S$.  In cases 4a and 4b, all vertices in the subtree rooted at $w_2$ were counted as belonging to $S$  because all boundary vertices were bottom vertices.  We now modify this calculation by  \emph{subtracting}   the number of  $\sbar$ vertices in this subtree.  Vertex $w_2$ and one, both, or neither of its subtrees are in $\sbar$, so we subtract  $2^{t_2} - c_2$ where $c_2 \in \{0,1\}$.  If $|\bnd(S)| = 2$ then $|S| = (2^{d+1} -1) - 2^{t_1} - 2^{t_2} + b$ with $0 \le b \le 2$.  This appears in equation 8.  If $|\bnd(S)| = 3$ then $|S| = (2^{d+1} -1) - 2^{t_1} - 2^{t_2} - 2^{t_3} + b$ with $0 \le b \le 3$.  This appears in equation 9. 

Now consider  case  5b in which the parent of $w_2$ is in $\sbar$.  In cases 4a and 4b, all vertices in the subtree rooted at $w_2$ were counted as belonging to $\sbar$ again because all boundary vertices were bottom vertices.  We now modify this calculation by  \emph{adding}   the number of  $S$ vertices in this subtree.  One or  both of  the subtrees rooted at $w_2$  are in $S$, so we add  $2^{t_2} - d_2$ where $d_2 \in \{1,2\}$.  If $|bnd(S)| = 2$ then $|S| = (2^{d+1} -1) - 2^{t_1} +2^{t_2} - b$ with $0 \le b \le 2$.  This appears in equation 10.
If $|\bnd(S)| = 3$ then $|S| = (2^{d+1} -1) - 2^{t_1} + 2^{t_2} - 2^{t_3} - b$ with $-1 \le b \le 2$.  This appears equation 11 (with the roles of $t_2$ and $t_3$ reversed).

\medskip 

\noindent {\bf Case 6: }  
 $r \in S$ and there there are exactly two vertices in $\bnd(S)$ that are not bottom vertices. 

In this case, $|\bnd(S)| = 3$ and there is exactly one bottom vertex.  Let $\bnd(S) = \{w_1,w_2,w_3\}$ where $w_1$ is the bottom vertex.  By Lemma~\ref{tree-lem}(d)  we know $w_1$ is on both the $w_2r$-path and the $w_3r$-path.  Without loss of generality we may assume that the $w_1w_2$-path does not contain $w_3$.

First consider the case in which the parent of $w_2$ is in $S$. If the parent of $w_3$ is also in $S$ then the subtree rooted at $w_3$ was counted as being in $S$ in case 5a, so we modify the calculation by \emph{subtracting}  the number of $\sbar$ vertices in this subtree.  Vertex $w_3$ and one, both, or neither of its subtrees are in $\sbar$, so we subtract $2^{t_3} - c_3$ where $c_3 \in \{0,1\}$ and get $|S| = 
  (2^{d+1} -1) - 2^{t_1} - 2^{t_2} - 2^{t_3} + b$ with $0 \le b \le 3$.   This appears in equation 9.  If instead, the parent of $w_3$ is in $\sbar$, then the subtree rooted at $w_3$ was counted as being in $\sbar$ in case 5a, so we modify the calculation by \emph{adding}   $2^{t_3} - d_3$ where $d_3 \in \{1,2\}$, as one or both subtrees of $w_3$ could be in $\sbar$.    The result is   $|S| = 
  (2^{d+1} -1) - 2^{t_1} - 2^{t_2} + 2^{t_3} - b$ with $-1 \le b \le 2$.  This appears in  equation 11 (with the roles of $t_2$ and $t_3$ reversed).  
 
 Now consider the case in which the parent of $w_2$ is in $\sbar$.  If the parent of $w_3$ is   in $S$ then the subtree rooted at $w_3$ was counted as being in $S$ in case 5b, so we modify the calculation by \emph{subtracting}   $2^{t_3} - c_3$ where $c_3 \in \{0,1\}$, as $w_3$ and one or both of its subtrees are in $\sbar$.  The result is   $|S| = 
  (2^{d+1} -1) - 2^{t_1} + 2^{t_2} - 2^{t_3} - b$ with $-1 \le b \le 2$.   This appears in equation 11.  If instead, the parent of $w_3$ is in $\sbar$, then the subtree rooted at $w_3$ was counted as being in $\sbar$ in case 5b, so we modify the calculation by \emph{adding}   $2^{t_3} - d_3$ where $d_3 \in \{1,2\}$, as one or both subtrees of $w_3$ must be in $S$.    The result is   $|S| = 
  (2^{d+1} -1) - 2^{t_1} + 2^{t_2} + 2^{t_3} - b$ with $1 \le b \le 4$.  This is appears in  equation 12.  
\end{proof}

\section{The $\gamma$-bound and the treatment number of $BT(8)$} \label{BT-8-sec}

In this section we introduce the $\gamma$-bound, a lower bound on the treatment number of a graph,  building on a result from \cite{FirstTreatPaper}.  We then use the $\gamma$-bound in our proof that  three vertices must be treated at some time-step 
 to clear the complete binary tree of depth $8$. The $\gamma$-bound appears again in Section~\ref{sec:K4s}
where we show  that this lower bound is not tight by constructing a family of graphs with arbitrarily large treatment number  where each satisfies the $2$-bound.

For a graph $H$ that is cleared by a  width $\gamma$ protocol, the set $S$ in our next definition represents a set of green vertices that remain green at a particular time-step $j$ without being re-treated, and $p$ represents a threshold for the number of green vertices that will be surpassed at the next time-step. 

\begin{defn} \label{gamma} \rm 
We say that a graph $H$ satisfies the \emph{$\gamma$-bound} if for each $p: \gamma + 1 \le p \le |V(H)| -1$ there exists a set $S \subseteq V(H)$ with $p - \gamma+ 1 \le |S| \le p$ so that $|\bnd(S)|  \le  2 \, \gamma -1$. 
\label{gamma-bound-def}
\end{defn}

Definition~\ref{gamma} is similar to the hypotheses of \cite[Theorem~4.4]{FirstTreatPaper}, which states that graphs with treatment number $\gamma$ must satisfy the $\gamma$-bound.  It  follows  directly from Definition~\ref{gamma-bound-def} and \cite[Theorem~4.4]{FirstTreatPaper} that a graph with treatment number at most $\gamma$ satisfies the $\gamma$-bound. We restate this below as Corollary~\ref{lower-bnd}.

\begin{cor}
If  $H$ is a graph with $\tau(H) \le \gamma$, then $H$ satisfies the $\gamma$-bound.  
\label{lower-bnd}
\end{cor}

The contrapositive of Corollary~\ref{lower-bnd} provides a lower bound on the treatment number, since a graph $H$ that fails to satisfy the $\gamma$-bound will have $\tau(H) > \gamma$.  
In our proof of the next theorem, 
we show that $BT(8)$ does not satisfy the $2$-bound by showing it is impossible  to surpass $167$ green vertices in $BT(8)$ when only two vertices are treated per time-step.

\begin{thm}\label{thm:depth8} The complete binary tree of depth 8 has treatment number 3. \end{thm}

\begin{proof}
From Corollary~\ref{cor:perfbin}, we know $\tau(BT(8)) \le 3$. Let $r$ be the root of $BT(8)$ and note that $|V(BT(8))| = 2^9 - 1 = 511$.   Suppose for a contradiction that $\tau(BT(8)) \le 2$ and fix a protocol $J$ of width 2 that clears $BT(8)$. 
By Corollary~\ref{lower-bnd},  the tree $BT(8)$ satisfies the $2$-bound, so for each $p$ with $3 \le p \le 510$, there exists a set $S \subseteq V(BT(8))$ with $p-1 \le |S| \le p$ so that $|\bnd(S)| \le 3$. Consequently, for  $p = 167$, there exists $S \subseteq V(BT(8))$ with $|S| \in \{166, 167\}$ and  $|\bnd(S)| \le 3$.  By Theorem~\ref{binary-tree-thm}, such a set $S$ with $|S| \in \{166,167\}$ and $|\bnd(S)|\leq 3$ satisfies one of the equations in Table~\ref{equns-table}.

For equations 1-6, we have $-1\leq b\leq 6$, and  thus $165 \le |S| + b \le 173.$   In each case, $|S| + b$ is the sum  and/or difference  of  powers of $2$ that are even, so $|S| + b \in \{166,168,170, 172\}$.    The binary representations of these integers are shown in Table~\ref{binary-exp}. Each integer, except for 168, has 4 or more $1$'s in its binary representation, and thus  $|S|$ cannot satisfy any of the equations 1-3 when $|S|+b \in \{166,167,172\}$. Now 168 has three $1$'s in its binary representation, but equation 3 is not satisfied when $|S|+b=168$ because $3 \leq b$ in equation 3 and thus we would have $|S|+3 \leq 168$, contradicting $S \in \{166,167\}$.
   
Thus $|S|$ does not satisfy any of equations 1-3 and we next consider equations 4-6.
In the binary representations of $166, 168, 170$,  and $172$, there are at least two gaps between $1$'s in the representation.  As a result, the binary representation of $|S| + b  + 2^{t_2}$ will have at least three $1$'s and the binary representation of $|S| + b  + 2^{t_2} + 2^{t_3}$  will have at least two $1$'s.  Consequently, $|S|$ does not satisfy any of equations 4, 5, 6.

For equations 7-12,  we have $(2^{d+1} -1) - |S| = 511 - |S| \in \{344, 345\}$.    In equations 7-9 we know $0 \le b \le 3$ and thus $344 \le 511-|S|  + b \le 348$.  In these three equations, the quantity $511 - |S| + b$ is the sum of at most three powers of $2$ that are even, so $ 511-|S|  + b \in \{344, 346, 348\}$.  The binary representations of these integers are shown in Table~\ref{binary-exp} and each has at least four powers of $2$ in its binary representation, so $|S|$ does not satisfy any of equations 7, 8, 9.

In equations 10-12, we have $-1 \le b \le 4$ and thus   $340 \le  511 - |S|  - b \le 346$.  Again, the quantity $511 - |S| - b$ must be even and the binary representation of the even integers between $340$ and $346$ are shown in Table~\ref{binary-exp}.  In the binary representations of $340, 342, 344, 346$, there are at least two gaps between $1$'s in each representation.  As a result, the binary representation of $511 - |S| - b  + 2^{t_2}$ will have at least three $1$'s and the binary representation of $511 - |S| - b  + 2^{t_2} + 2^{t_3}$  will have at least two $1$'s.  Consequently, $|S|$ does not satisfy any of equations 10, 11, 12.

\begin{table}[ht] \small 
    \centering
    \begin{tabular}{c|rr||c|rr} \hline 
166  & $128+32+4+2$ & $10100110$ & 340 & $256+64+16+4$ & $101010100$\\
168 & $128+32+8$ & $10101000$ & 342 & $256+64+16+4+2$ & $101010110$\\
170 & $128+32+8+2$ & $10101010 $ & 344 & $256+64+16+8$& $101011000$\\
172 & $128+32+8+4$ & $10101100 $  & 346 & $256+64 + 16+8+2$& $101011010$ \\
&&& 348 & $256+64+16 +8+4$& $101011100$\\ \hline 
    \end{tabular}
    \caption{Binary representations of numbers that arise from Equations 1-12 in Table~\ref{equns-table} 
    when $|S| \in \{166,167\}$.}
    \label{binary-exp}
\end{table}
\end{proof}

We conclude this section by presenting a subgraph result \cite[Corollary 2.10] {FirstTreatPaper} in Proposition~\ref{subgraph}
 and then applying it to determine the treatment number of $BT(9)$ and $BT(10)$.

\begin{prop} [\cite{FirstTreatPaper}] \label{subgraph} If $H'$ is a subgraph of $H$, then $\tau(H')\leq \tau(H)$. 
\end{prop} 

\begin{cor}
    The treatment number of $BT(d)$ is $3$ when $d \in \{9,10\}$.
\end{cor}

\begin{proof}
Let $d \in \{9,10\}$.  We know $BT(8)$ is a subgraph of $BT(d)$ and $\tau(BT(8)) = 3$ by Theorem~\ref{thm:depth8},
thus $BT(d) \ge BT(8) \ge 3$.  The reverse inequality $BT(d) \le 3$ follows from Corollary~\ref{cor:perfbin}.  
\end{proof}

\section{Treatment number of $BT(7)$} \label{sec:depth-seven-tree}

 This section focuses on the complete binary tree of depth $7$, and we denote $BT(7)$ by $T$.  
 We are able to construct a protocol
 $\mathcal{P}$ of width two that clears $T$, and  it  consists of a sequence of eight sub-protocols.  For a graph $H$, we say that a sub-protocol ${\cal P}'$ {\it clears} a subgraph $H'$ of $H$ if all vertices of $H'$ are green after the last step of ${\cal P}'$. Figure~\ref{fig:T7} provides a high-level view of $T$,  with the root labeled  $r$ and 
  labels specified for the vertices at depths $1$, $2$ and $3$.  The remaining vertices  shown are contained in the depth $4$ complete binary trees $T_{z_i}$, rooted at $z_i$ for $1 \leq i \leq 8$.
   Figure~\ref{fig:T7}  also shows which of sub-protocols ${\cal P}_1$, ${\cal P}_2$, ${\cal P}_3$ is used to clear $T_{z_i}$ for each $i: 1 \leq i \leq 8$.  In forming $\cal P$, the sub-protocols will be applied in the order shown, for example, $\mathcal{P}_1$ is applied to clear $T_{z_1}$, then $\mathcal{P}_1$ is applied to clear $T_{z_2}$, then $\mathcal{P}_2$ is applied to clear $T_{z_3}$ and so on.

  Each of  Figures~\ref{fig:proto1}, \ref{fig:proto2}, and \ref{fig:proto3} provides a view of  $T$ that focuses on   subtree $T_{z_i}$ and shows the steps in the sub-protocols ${\cal P}_1$, ${\cal P}_2$, ${\cal P}_3$, respectively.  
In these figures, we label subtrees $T_w, T_x, T_y$, so that we can refer to them in  proofs.   Table~\ref{table:8sub} provides the sequence of the eight sub-protocols that form $\mathcal{P}$ and also the mapping between the labeling of vertices of $T$ in Figure~\ref{fig:T7} and that in Figures~\ref{fig:proto1}, \ref{fig:proto2}, and \ref{fig:proto3}. For example, when ${\cal P}_1$ is used to clear $T_{z_1}$ then $z_i = z_1$, $u = u_1$, $v = v_1$, etc. as in the first row of Table~\ref{table:8sub}, and when ${\cal P}_1$ is used to clear $T_{z_8}$ then $z_i = z_8$, $u = u_4$, $v = v_2$, etc. as in the last row of Table~\ref{table:8sub}. 
In sub-protocols ${\cal P}_2$ and ${\cal P}_3$  (Figures~\ref{fig:proto2} and \ref{fig:proto3}), some vertices will turn from green to yellow to red and this is indicated in the figures by a red boxed number.  The labels listed in Figures~\ref{fig:proto1}, \ref{fig:proto2}, and \ref{fig:proto3}  indicate the time-steps $t'$ within the relevant sub-protocol ${\cal P}_i$.  When the sub-protocols are combined to create protocol $\cal P$, these labels will shift.  For example, vertex $\beta_1$  is treated at time-steps $3$, $5$, and $7$ within ${\cal P}_1$, but when ${\cal P}_1$ is used starting at time-step $t$ then 
$\beta_1$  is treated at time-steps $t+3$, $t+5$, and $t+7$. 

\begin{figure}[htp] 
\centering 
\begin{tikzpicture}[scale=.85]
\draw[thick] (4,2)--(8,4)--(12,2);
\draw[thick] (5.1,2)--(5.5,2.75);
\draw[thick] (6.2,2)--(7.2,2.8)--(7.4,2);
\draw[thick] (8.5,2)--(8.8,2.8)--(9.7,2);
\draw[thick] (8.8,2.75)--(9,3.5);
\draw[thick] (7.2, 2.75)--(7,3.5);
\draw[thick] (10.8, 2)--(10.5,2.75);

\draw[gray, thick] (4,2)--(3.6,0)--(4.3,0)--(4,2);
\draw[gray, thick] (5.1,2)--(4.7,0)--(5.4,0)--(5.1,2);
\draw[gray, thick] (6.2,2)--(5.9,0)--(6.6,0)--(6.2,2);
\draw[gray, thick] (7.4,2)--(7,0)--(7.7,0)--(7.4,2);

\draw[gray, thick] (8.5,2)--(8.2,0)--(8.9,0)--(8.5,2);
\draw[gray, thick] (9.7,2)--(9.3,0)--(10,0)--(9.7,2);
\draw[gray, thick] (10.8,2)--(10.5,0)--(11.2,0)--(10.8,2);
\draw[gray, thick] (12,2)--(11.6,0)--(12.3,0)--(12,2);

\filldraw[black]
(8,4) circle [radius=3pt]
(7,3.5) circle [radius=3pt]
(9,3.5) circle [radius=3pt]
(5.5,2.75) circle [radius=3pt]
(7.2,2.8) circle [radius=3pt]
(8.8,2.8) circle [radius=3pt]
(10.5,2.75) circle [radius=3pt]
(4,2) circle [radius=3pt]
(5.1,2) circle [radius=3pt]
(6.2,2) circle [radius=3pt]
(7.4,2) circle [radius=3pt]
(8.5,2) circle [radius=3pt]
(9.7,2) circle [radius=3pt]
(10.8,2) circle [radius=3pt]
(12,2) circle [radius=3pt]
;

\node at (7.7,4.1) {\small $r$};
\node at (6.7,3.6) {\small $v_1$};
\node at (8.6,3.5) {\small $v_2$};

\node at (5.15,2.8) {\small $u_1$};
\node at (6.8,2.8) {\small $u_2$};
\node at (8.4,2.8) {\small $u_3$};
\node at (10.15,2.75) {\small $u_4$};

\node at (3.7,2) {\small $z_1$};
\node at (4.8,2) {\small $z_2$};
\node at (5.9,2) {\small $z_3$};
\node at (7.1,2) {\small $z_4$};
\node at (8.2,2) {\small $z_5$};
\node at (9.4,2) {\small $z_6$};
\node at (10.5,2) {\small $z_7$};
\node at (11.6,2) {\small $z_8$};

\node at (4,.25) {\small $T_{z_1}$};
\node at (5.1,.25) {\small $T_{z_2}$};
\node at (6.3,.25) {\small $T_{z_3}$};
\node at (7.4,.25) {\small $T_{z_4}$};
\node at (8.6,.25) {\small $T_{z_5}$};
\node at (9.7,.25) {\small $T_{z_6}$};
\node at (10.87,.25) {\small $T_{z_7}$};
\node at (12,.25) {\small $T_{z_8}$};

\node at (4,-.5) {\small \color{blue} $\mathcal P_1$};
\node at (5.1,-.5) {\small \color{blue} $\mathcal P_1$};
\node at (6.2,-.5) {\small \color{blue} $\mathcal P_2$};
\node at (7.35,-.5) {\small \color{blue} $\mathcal P_3$};
\node at (8.6,-.5) {\small \color{blue} $\mathcal P_2$};
\node at (9.7,-.5) {\small \color{blue} $\mathcal P_3$};
\node at (10.9,-.5) {\small \color{blue} $\mathcal P_1$};
\node at (12,-.5) {\small \color{blue} $\mathcal P_1$};

\end{tikzpicture}
\caption{A high-level view of the tree $BT(7)$ showing which sub-protocol ${\cal P}_1, $ ${\cal P}_2, $ or ${\cal P}_3$ is used to clear each of the depth $4$ trees rooted at $z_i$.} 
\label{fig:T7}
\end{figure}
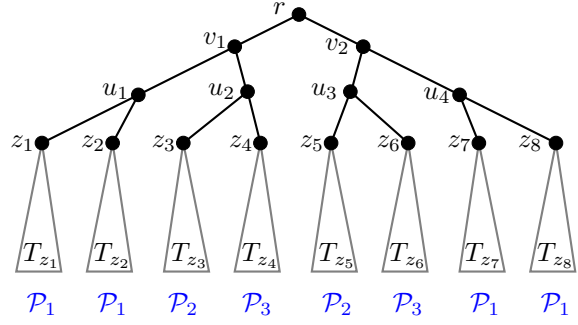

The next lemma specifies a set of vertices that will be green after sub-protocol $\mathcal{P}_1$ is applied to tree $T$.  Note that in applying sub-protocol $\mathcal{P}_1$, we make no assumptions about the color of any of the vertices that are not treated in the sub-protocol.  

\begin{lemma}\label{lem1} 
The sub-protocol ${\cal P}_1$, shown in Figure~\ref{fig:proto1}, has width $2$ and after it is applied to tree $T$, the vertices in $T_{z_i}$ are all green.  Furthermore, in $\mathcal{P}_1$, vertex $u$ is treated   at every other time-step including the last one. 
\end{lemma} 

\begin{proof} 
There are $33$ time-steps $t'$  in ${\cal P}_1$  and vertex $u$ is treated at each odd value of $t'$.  Each of the numbers $1$, $2$, \ldots, $33$ appears exactly twice, so ${\cal P}_1$ has width $2$.    It remains to show that ${\cal P}_1$ clears the vertices in $T_{z_i}$.  First, the $7$ vertices in the left subtree of $\alpha_1$ are cleared, then $\alpha_1$ is treated at every other time-step until its right subtree is cleared.  At   $t'=17$, the left subtree of $z_i$ is clear.  Vertex $z_i$ is not turned red by $u$ since $u$ is treated at every other time-step.  Likewise, is is not turned red by $\alpha_2$ because $\alpha_2$ is first treated at $t'=18$ and then again in every other time-step until $t'=26$.  By time-step $26$ the subtree rooted at $\beta_3$ is clear.  Starting at $t'=27$, vertex $\beta_4$ is treated during every other time-step until $t'=31$, and its subtree is clear at $t'=33$.  Thus all vertices in the subtree rooted at $z_i$ are clear by $t'=33$. Indeed, in  ${\cal P}_1$, no vertex turns red once it is green.
\end{proof}

\begin{table}[htbp] \small 
\[ \begin{tabular}{|c|c||c|c|c|c|c|c|}
\hline
stage & sub-protocol & $z_i$ & $u$ & $y$ & $v$ & $w$ & $x$ \\
\hline
1. & $\mathcal{P}_1$ & $z_1$ & $u_1$ & $z_2$ & $v_1$ & $u_2$ & $r$ \\
2. & $\mathcal{P}_1$ & $z_2$ & $u_1$ & $z_1$ & $v_1$ & $u_2$ & $r$ \\
3. & $\mathcal{P}_2$ & $z_3$ & $u_2$ & $z_4$ & $v_1$ & $u_1$ & $r$ \\
4. & $\mathcal{P}_3$ & $z_4$ & $u_2$ & $z_3$ & $v_1$ & $u_1$ & $r$ \\
5. & $\mathcal{P}_2$ & $z_5$ & $u_3$ & $z_6$ & $v_2$ & $r$ & $u_4$ \\
6. & $\mathcal{P}_3$ & $z_6$ & $u_3$ & $z_5$ & $v_2$ & $r$ & $u_4$ \\
7. & $\mathcal{P}_1$ & $z_7$ & $u_4$ & $z_8$ & $v_2$ & $r$ & $u_3$ \\
8. & $\mathcal{P}_1$ & $z_8$ & $u_4$ & $z_7$ & $v_2$ & $r$ & $u_3$ \\
\hline
\end{tabular} \]

\caption{The protocol $\mathcal{P}$ consisting of this sequence of 8 sub-protocols, and the values of $z_i,u,y,v,w$, and $x$, for each  sub-protocol in the sequence.}
\label{table:8sub}
\end{table}

Sub-protocols $\mathcal{P}_2$ and $\mathcal{P}_3$ are each applied when there is a particular set of existing green vertices in tree $T$.  For sub-protocol $\mathcal{P}_2$, the vertices in $T_w$ are initially green, as indicated by the shaded portion of Figure~\ref{fig:proto2}.  Similarly, for sub-protocol $\mathcal{P}_3$, the vertices $x,v,u,z_i$, as well as those in $T_w$ and $T_y$ are all initially green, as shown in Figure~\ref{fig:proto3}. Lemmas~\ref{lem2} and~\ref{lem3} specify which vertices must be green  initially and which vertices will be green at the last time-step of these protocols.

\begin{figure}[htb]
\centering
\begin{tikzpicture}[scale=.83]
\draw[thick] (-.3,0)--(.5,1)--(1,0); 
\draw[thick] (2,0)--(2.5,1)--(3.3,0); 
\draw[thick] (.5,1)--(1.5,1.8)--(2.5,1); 

\draw[thick] (4.7,0)--(5.5,1)--(6,0); 
\draw[thick] (7,0)--(7.5,1)--(8.3,0); 
\draw[thick] (5.5,1)--(6.5,1.8)--(7.5,1); 

\draw[thick] (9.7,0)--(10.5,1)--(11,0); 
\draw[thick] (12,0)--(12.5,1)--(13.3,0); 
\draw[thick] (10.5,1)--(11.5,1.8)--(12.5,1); 

\draw[thick] (14.7,0)--(15.5,1)--(16,0); 
\draw[thick] (17,0)--(17.5,1)--(18.3,0); 
\draw[thick] (15.5,1)--(16.5,1.8)--(17.5,1); 

\draw[thick] (16.5,1.8)--(14,3)--(11.5,1.8);

\draw[thick] (6.5,1.8)--(4,3)--(1.5,1.8);

\draw[thick] (14,3)--(9,4)--(4,3);
\draw[thick] (9,4)--(9,7);
\draw[thick] (9,6)--(11.7,6);
\draw[thick] (7,8.5)--(9,7)--(11,8.5);

\draw (12,8.5) ellipse (1.5 and 0.5);
\draw (6,8.5) ellipse (1.5 and 0.5);
\draw (12.8,6) ellipse (1.5 and 0.5);

\node(ell9) at (4,8.6) {$T_w$};
\node(ell9) at (14,8.6) {$T_x$};
\node(ell9) at (14.8,6.1) {$T_y$};

\filldraw[black]
(11,8.5) circle [radius=3pt]
(11.7,6) circle [radius=3pt]
(7,8.5) circle [radius=3pt]
(9,7) circle [radius=3pt]
(9,6) circle [radius=3pt]
(9,4) circle [radius=3pt]
(4,3) circle [radius=3pt]
(14,3) circle [radius=3pt]

(11.5,1.8) circle [radius=3pt]
(12.5,1) circle [radius=3pt]
(10.5,1) circle [radius=3pt]
(13.3,0) circle [radius=3pt]
(12,0) circle [radius=3pt]
(11,0)circle [radius=3pt]
(9.7,0)circle [radius=3pt]

(1.5,1.8) circle [radius=3pt]
(2.5,1) circle [radius=3pt]
(.5,1) circle [radius=3pt]
(3.3,0) circle [radius=3pt]
(2,0) circle [radius=3pt]
(1,0)circle [radius=3pt]
(-.3,0)circle [radius=3pt]
;
\filldraw[black]
(16.5,1.8) circle [radius=3pt]
(15.5,1) circle [radius=3pt]
(17.5,1) circle [radius=3pt]
(18.3,0) circle [radius=3pt]
(17,0) circle [radius=3pt]
(16,0)circle [radius=3pt]
(14.7,0)circle [radius=3pt]

(6.5,1.8) circle [radius=3pt]
(5.5,1) circle [radius=3pt]
(7.5,1) circle [radius=3pt]
(8.3,0) circle [radius=3pt]
(7,0) circle [radius=3pt]
(6,0)circle [radius=3pt]
(4.7,0)circle [radius=3pt]
;
\node(ell8) at (6.6,8.5) {$w$};
\node(ell8) at (11.4,8.5) {$x$};
\node(ell8) at (12.2,6) {$y$};

\node(ell4) at (9,7.3) {$v$};
\node(ell4) at (9.3,5.7) {$u$};
\node(ell0) at (4.4,6) {\color{green!75!black}  \footnotesize $[1,3,5,7,9,11,13,15,17,19,21,23,25,27,29,31,33]$};
\node(ell4) at (8.6,4.2) {$z_i$};
\node(ell0) at (9.4,4.2) {\color{green!75!black}  \footnotesize $[17]$};
\node(ell3) at (13.1,3) {$\alpha_2$};
\node(ell0) at (16,3) {\color{green!75!black}  \footnotesize $[18,20,22,24,26]$};
\node(ell3) at (3.4,3) {$\alpha_1$};
\node(ell0) at (6.1,2.9) {\color{green!75!black}  \footnotesize $[8,10,12,14,16]$};

\node(ell2) at (15.9,1.8) {$\beta_4$};
\node(ell2) at (17.55,2) {\color{green!75!black} \small $[27,29,31]$};
\node(ell1) at (14.8,1.2) {{\color{green!75!black}\small  $[28,30]$}};
\node(ell1) at (18,1.2) {\color{green!75!black} \small $[32]$};
\node(ell0) at (14.7,-.4) {\color{green!75!black}  \small$[28]$};
\node(ell0) at (16.05,-.4) {\color{green!75!black} \small $[30]$};
\node(ell0) at (17,-.4) {\color{green!75!black} \small $[32]$};
\node(ell0) at (18.3,-.4) {\color{green!75!black}  \small $[33]$};

\node(ell2) at (12.2,1.8) {$\beta_3$};
\node(ell2) at (10.9,2) {\color{green!75!black} \small $[22]$};
\node(ell1) at (9.75,1.2) {{\color{green!75!black}\small  $[19,21]$}};
\node(ell1) at (13.2,1.2) {\color{green!75!black} \small $[23,25]$};
\node(ell0) at (9.7,-.4) {\color{green!75!black}  \small$[18]$};
\node(ell0) at (11.05,-.4) {\color{green!75!black} \small $[20]$};
\node(ell0) at (12.05,-.4) {\color{green!75!black} \small $[24]$};
\node(ell0) at (13.3,-.4) {\color{green!75!black}  \small $[26]$};

\node(ell2) at (5.9,1.8) {$\beta_2$};
\node(ell2) at (6.95,2) {\color{green!75!black} \small $[12]$};
\node(ell1) at (4.8,1.2) {{\color{green!75!black}\small  $[9,11]$}};
\node(ell1) at (8.2,1.2) {\color{green!75!black} \small $[13,15]$};
\node(ell0) at (4.7,-.4) {\color{green!75!black}  \small$[8]$};
\node(ell0) at (6.05,-.4) {\color{green!75!black} \small $[10]$};
\node(ell0) at (7.05,-.4) {\color{green!75!black} \small $[14]$};
\node(ell0) at (8.3,-.4) {\color{green!75!black}  \small $[16]$};

\node(ell2) at (2.1,1.8) {$\beta_1$};
\node(ell2) at (.7,2) {\color{green!75!black} \small $[3,5,7]$};
\node(ell1) at (.1,1.2) {{\color{green!75!black}\small  $[2]$}};
\node(ell1) at (3.1,1.2) {\color{green!75!black} \small $[4,6]$};
\node(ell0) at (-.3,-.4) {\color{green!75!black}  \small$[1]$};
\node(ell0) at (1.05,-.4) {\color{green!75!black} \small $[2]$};
\node(ell0) at (2.05,-.4) {\color{green!75!black} \small $[4]$};
\node(ell0) at (3.3,-.4) {\color{green!75!black}  \small $[6]$};

\end{tikzpicture}

\caption{Sub-protocol ${\cal P}_1$ applied to  $T$. The green number(s) next to a vertex $q$ indicate the time-step(s) in ${\cal P}_1$ at which $q$ is treated.}

\label{fig:proto1}

\end{figure}
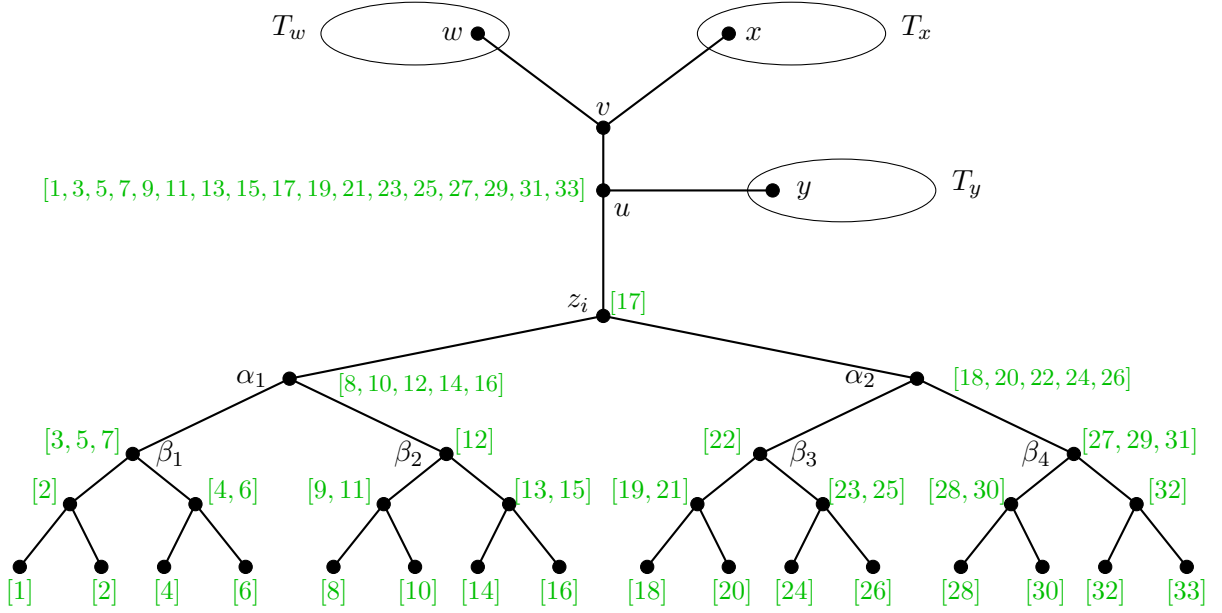

\begin{lemma}\label{lem2} 
Let $T_w,T_x,T_y$, be the subtrees of $T$ shown in Figure~\ref{fig:proto2} and let  ${\cal P}_2$ the  sub-protocol also shown in that figure.  
Further, suppose that at time $0$ of ${\cal P}_2$, the vertices in $T_w$ are green.      
Then  ${\cal P}_2$ has width $2$ and when it is applied 
 to $T$, at the last time-step, vertices $x$ and $y$ are treated and the vertices in $T_w,T_{z_i}$ and $\{u,v,x,y\}$ are also green.
\end{lemma}

\begin{proof} 
There are $45$ time-steps $t'$  in ${\cal P}_2$ shown in Figure~\ref{fig:proto2} and vertices $x$ and $y$ are treated at time-step $45$.
Each of  the unboxed  numbers $1$, $2$, \ldots, $45$ appears exactly twice, so ${\cal P}_2$ has width $2$. 
 It remains to verify that at time-step $45$, all vertices shown in the figure are green, except possibly those in $T_x$ (other than $x$); and those in $T_y$ (other than $y$).

Without loss of generality, we may assume that at time $0$ of ${\cal P}_2$, the vertices other than those in $T_w$ are red.  As in sub-protocol ${\cal P}_1$, the vertices in the subtree rooted at $\alpha_1$ are green when $t'=16$.   Vertex $z_i$ is treated at time-steps $17, 18, 20$ and $22$, and turns red when $t' = 24$, causing $\alpha_1$ to turn yellow when $t' = 24$.  However, $\alpha_1$ is treated at time-steps $25$ and $27$ and   we will show that vertex $z_i$ is green or yellow starting at  $t' = 27$, so the entire subtree rooted at $\alpha_1$ will be  green at the end of ${\cal P}_2$.    One can check that the vertices in the subtree rooted at $\alpha_2$ never turn red once they are green, and all are green by $t'= 42$, thus vertex $z_i$ does not turn red from its adjacency to $\alpha_2$ after $t' = 29$.   

Vertex $v$ turns red  at time-steps $19$, $25$, and $38$ (due to its adjacency with red neighbors $x$ and $u$), 
so $w$ turns yellow at those time steps, but is treated when $t' = 20, 26, 39, 41, 43$, so $w$ is never red and the other vertices in $T_w$ remain green throughout sub-protocol ${\cal P}_2$.    Vertex $v$ is treated at time-step $44$, so $w$ remains green after it is treated at $t' = 43$.  Vertex $u$ is treated at $t' = 43$, and its neighbors $v$ and $y$ are  green at $t' = 44, 45$, so $u$ is green at time-step $45$ unless $z_i$ is red at time-step $44$ or $45$.  Vertex $z_i$ is treated at $t' = 27, 29, 31, 33$, and it does not   turn yellow again until $t' = 37$,  but then it is treated at $t' = 38, 40, 42$, so it does not turn red between $t' = 27$ and $t' = 42$.  After this, its neighbors are green, so it remains green at time-steps $43$, $44$, and $45$, and thus remains green at the end of  ${\cal P}_2$.

Thus, at the conclusion, the only vertices in $T$ that may not be green, are those in $T_x$ (other than $x$ and those in $T_y$ (other than $y$).
\end{proof}

We omit the proof of Lemma~\ref{lem3} since it is similar to that of Lemma~\ref{lem2}.

\begin{lemma}\label{lem3} Let $T$ be the tree and $T_w$, $T_x$, $T_y$ be the subtrees shown in Figure~\ref{fig:proto3} and let $\mathcal{P}_3$ be the sub-protocol also shown in that figure. Further, suppose that at time $0$ of ${\cal P}_3$, vertices in subtrees $T_w$, $T_y$, along with $\{u,v,x,z_i\}$ are green.  Then $\mathcal{P}_3$ has width $2$ and when it is applied to $T$, at the last time-step, vertex $x$ is treated and the vertices in $T_{z_i}$, $T_w$, $T_y$ and $\{u,v\}$ are also green.
\end{lemma}

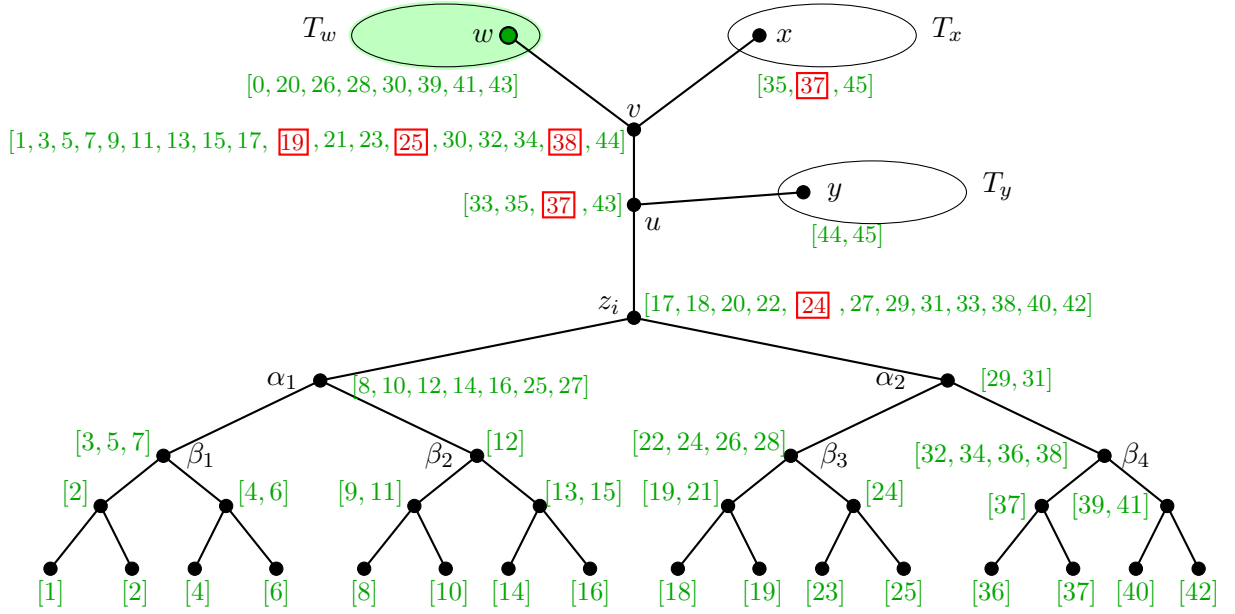
\begin{figure}[htb]
\centering
\begin{tikzpicture}[scale=.83]
\filldraw[green!25!white]  (6,8.5) ellipse (1.54 and 0.54);
\draw[thick] (-.3,0)--(.5,1)--(1,0); 
\draw[thick] (2,0)--(2.5,1)--(3.3,0); 
\draw[thick] (.5,1)--(1.5,1.8)--(2.5,1); 

\draw[thick] (4.7,0)--(5.5,1)--(6,0); 
\draw[thick] (7,0)--(7.5,1)--(8.3,0); 
\draw[thick] (5.5,1)--(6.5,1.8)--(7.5,1); 

\draw[thick] (9.7,0)--(10.5,1)--(11,0); 
\draw[thick] (12,0)--(12.5,1)--(13.3,0); 
\draw[thick] (10.5,1)--(11.5,1.8)--(12.5,1); 

\draw[thick] (14.7,0)--(15.5,1)--(16,0); 
\draw[thick] (17,0)--(17.5,1)--(18,0); 
\draw[thick] (15.5,1)--(16.5,1.8)--(17.5,1); 

\draw[thick] (16.5,1.8)--(14,3)--(11.5,1.8);

\draw[thick] (6.5,1.8)--(4,3)--(1.5,1.8);

\draw[thick] (14,3)--(9,4)--(4,3);
\draw[thick] (9,4)--(9,7);
\draw[thick] (9,5.8)--(11.7,6);
\draw[thick] (7,8.5)--(9,7)--(11,8.5);

\draw (12,8.5) ellipse (1.5 and 0.5);
\draw (6,8.5) ellipse (1.5 and 0.5);
\draw (12.8,6) ellipse (1.5 and 0.5);

\node(ell9) at (4,8.6) {$T_w$};
\node(ell9) at (14,8.6) {$T_x$};
\node(ell9) at (14.8,6.1) {$T_y$};

\filldraw[black]
(11,8.5) circle [radius=3pt]
(11.7,6) circle [radius=3pt]
(7,8.5) circle [radius=4pt]
(9,7) circle [radius=3pt]
(9,5.8) circle [radius=3pt]
(9,4) circle [radius=3pt]
(4,3) circle [radius=3pt]
(14,3) circle [radius=3pt]

(11.5,1.8) circle [radius=3pt]
(12.5,1) circle [radius=3pt]
(10.5,1) circle [radius=3pt]
(13.3,0) circle [radius=3pt]
(12,0) circle [radius=3pt]
(11,0)circle [radius=3pt]
(9.7,0)circle [radius=3pt]

(1.5,1.8) circle [radius=3pt]
(2.5,1) circle [radius=3pt]
(.5,1) circle [radius=3pt]
(3.3,0) circle [radius=3pt]
(2,0) circle [radius=3pt]
(1,0)circle [radius=3pt]
(-.3,0)circle [radius=3pt]
;
\filldraw[green!65!black]
(7,8.5) circle [radius=3pt];
\filldraw[black]
(16.5,1.8) circle [radius=3pt]
(15.5,1) circle [radius=3pt]
(17.5,1) circle [radius=3pt]
(18,0) circle [radius=3pt]
(17,0) circle [radius=3pt]
(16,0)circle [radius=3pt]
(14.7,0)circle [radius=3pt]

(6.5,1.8) circle [radius=3pt]
(5.5,1) circle [radius=3pt]
(7.5,1) circle [radius=3pt]
(8.3,0) circle [radius=3pt]
(7,0) circle [radius=3pt]
(6,0)circle [radius=3pt]
(4.7,0)circle [radius=3pt]
;
\node(ell8) at (6.6,8.5) {$w$};
\node(ell7) at (5,7.7) {\color{green!65!black}  \footnotesize $[0,20,26,28,30,39,41,43]$};
\node(ell8) at (11.4,8.5) {$x$};
\node(ell7) at (11.25,7.7) {\color{green!65!black}  \footnotesize $[35,$};
\node(ell7) at (11.85,7.7) {\color{red!85!black}  \footnotesize $37$};
\draw[red, thick] (11.6,7.5) rectangle (12.08,7.9);
\node(ell7) at (12.5,7.7) {\color{green!65!black}  \footnotesize $,45]$};

\node(ell8) at (12.2,6) {$y$};
\node(ell7) at (12.4,5.3) {\color{green!65!black}  \footnotesize $[44,45]$};

\node(ell4) at (9,7.3) {$v$};
\node(ell5) at (1.1,6.8) {\color{green!65!black}  \footnotesize $[1,3,5,7,9,11,13,15,17,$};
\node(ell5) at (3.55,6.8) {\color{red!85!black}  \footnotesize $19$};
\draw[red, thick] (3.35,6.6) rectangle (3.8,7);
\node(ell5) at (4.5,6.8) {\color{green!65!black}  \footnotesize $,21,23,$};
\node(ell5) at (5.45,6.8) {\color{red!85!black}  \footnotesize $25$};
\draw[red, thick] (5.2,6.6) rectangle (5.70,7);
\node(ell5) at (6.68,6.8) {\color{green!65!black}  \footnotesize $,30,32,34,$};
\node(ell5) at (7.9,6.8) {\color{red!85!black}  \footnotesize $38$};
\draw[red, thick] (7.65,6.6) rectangle (8.15,7);
\node(ell5) at (8.55,6.8) {\color{green!65!black}  \footnotesize $,44]$};

\node(ell4) at (9.3,5.5) {$u$};
\node(ell0) at (6.85,5.8) {\color{green!65!black}  \footnotesize $[33,35,$};
\draw[red, thick] (7.55,5.6) rectangle (8.05,6);
\node(ell0) at (7.8,5.8) {\color{red!85!black}  \footnotesize $37$};
\node(ell0) at (8.5,5.8) {\color{green!65!black}  \footnotesize $,43]$};

\node(ell4) at (8.6,4.2) {$z_i$};

\node(ell0) at (10.3,4.2) {\color{green!65!black}  \footnotesize $[17,18,20,22,$};
\node(ell0) at (11.87,4.2) {\color{red!85!black}  \footnotesize $24$};
\draw[red, thick] (11.6,4) rectangle (12.1,4.4);
\node(ell0) at (14.3,4.2) {\color{green!65!black}  \footnotesize $,27,29,31,33,38,40,42]$};

\node(ell3) at (13.1,3) {$\alpha_2$};
\node(ell0) at (15.1,3) {\color{green!65!black}  \footnotesize $[29,31]$};
\node(ell3) at (3.4,3) {$\alpha_1$};
\node(ell0) at (6.4,2.9) {\color{green!65!black}  \footnotesize $[8,10,12,14,16,25,27]$};

\node(ell2) at (17,1.8) {$\beta_4$};
\node(ell2) at (14.7,1.8) {\color{green!65!black} \small $[32,34,36,38]$};
\node(ell1) at (14.95,1) {{\color{green!65!black}\small  $[37]$}};
\node(ell1) at (16.6,1) {\color{green!65!black} \small $[39,41]$};
\node(ell0) at (14.7,-.4) {\color{green!65!black}  \small$[36]$};
\node(ell0) at (16.05,-.4) {\color{green!65!black} \small $[37]$};
\node(ell0) at (17,-.4) {\color{green!65!black} \small $[40]$};
\node(ell0) at (18,-.4) {\color{green!65!black}  \small $[42]$};

\node(ell2) at (12.2,1.8) {$\beta_3$};
\node(ell2) at (10.2,2) {\color{green!65!black} \small $[22,24,26,28]$};
\node(ell1) at (9.75,1.2) {{\color{green!65!black}\small  $[19,21]$}};
\node(ell1) at (13,1.2) {\color{green!65!black} \small $[24]$};
\node(ell0) at (9.7,-.4) {\color{green!65!black}  \small$[18]$};
\node(ell0) at (11.05,-.4) {\color{green!65!black} \small $[19]$};
\node(ell0) at (12.05,-.4) {\color{green!65!black} \small $[23]$};
\node(ell0) at (13.3,-.4) {\color{green!65!black}  \small $[25]$};

\node(ell2) at (5.9,1.8) {$\beta_2$};
\node(ell2) at (6.95,2) {\color{green!65!black} \small $[12]$};
\node(ell1) at (4.8,1.2) {{\color{green!65!black}\small  $[9,11]$}};
\node(ell1) at (8.2,1.2) {\color{green!65!black} \small $[13,15]$};
\node(ell0) at (4.7,-.4) {\color{green!65!black}  \small$[8]$};
\node(ell0) at (6.05,-.4) {\color{green!65!black} \small $[10]$};
\node(ell0) at (7.05,-.4) {\color{green!65!black} \small $[14]$};
\node(ell0) at (8.3,-.4) {\color{green!65!black}  \small $[16]$};

\node(ell2) at (2.1,1.8) {$\beta_1$};
\node(ell2) at (.7,2) {\color{green!65!black} \small $[3,5,7]$};
\node(ell1) at (.1,1.2) {{\color{green!65!black}\small  $[2]$}};
\node(ell1) at (3.1,1.2) {\color{green!65!black} \small $[4,6]$};
\node(ell0) at (-.3,-.4) {\color{green!65!black}  \small$[1]$};
\node(ell0) at (1.05,-.4) {\color{green!65!black} \small $[2]$};
\node(ell0) at (2.05,-.4) {\color{green!65!black} \small $[4]$};
\node(ell0) at (3.3,-.4) {\color{green!65!black}  \small $[6]$};

\end{tikzpicture}

\caption{Sub-protocol ${\cal P}_2$ applied to $T$. The green number(s) next to a vertex $q$ indicate the time-step(s) in ${\cal P}_2$ at which $q$ is treated and a red boxed number $t'$ next to $q$ indicates that $q$ could be red at time-step $t'$ after being yellow at time-step $t'-1$.}

\label{fig:proto2}
\end{figure}

\begin{figure}[htb]
\centering
\begin{tikzpicture}[scale=.83]
\filldraw[green!25!white]  (6,8.5) ellipse (1.54 and 0.54);
\filldraw[green!25!white]  (12.8,6) ellipse (1.54 and 0.54);
\filldraw[green!25!white]  (6,8.5) ellipse (1.54 and 0.54);

\draw[thick] (-.3,0)--(.5,1)--(1,0); 
\draw[thick] (2,0)--(2.5,1)--(3.3,0); 
\draw[thick] (.5,1)--(1.5,1.8)--(2.5,1); 

\draw[thick] (4.7,0)--(5.5,1)--(6,0); 
\draw[thick] (7,0)--(7.5,1)--(8.3,0); 
\draw[thick] (5.5,1)--(6.5,1.8)--(7.5,1); 

\draw[thick] (9.7,0)--(10.5,1)--(11,0); 
\draw[thick] (12,0)--(12.5,1)--(13.3,0); 
\draw[thick] (10.5,1)--(11.5,1.8)--(12.5,1); 

\draw[thick] (14.7,0)--(15.5,1)--(16,0); 
\draw[thick] (17,0)--(17.5,1)--(18,0); 
\draw[thick] (15.5,1)--(16.5,1.8)--(17.5,1); 

\draw[thick] (16.5,1.8)--(14,3)--(11.5,1.8);

\draw[thick] (6.5,1.8)--(4,3)--(1.5,1.8);

\draw[thick] (14,3)--(9,4)--(4,3);
\draw[thick] (9,4)--(9,7);
\draw[thick] (9,5.8)--(11.7,6);
\draw[thick] (7,8.5)--(9,7)--(11,9.5);

\draw (12,9.4) ellipse (1.5 and 0.5);
\draw (6,8.5) ellipse (1.5 and 0.5);
\draw (12.8,6) ellipse (1.5 and 0.5);

\node(ell9) at (4,8.6) {$T_w$};
\node(ell9) at (14,9.6) {$T_x$};
\node(ell9) at (14.8,6.1) {$T_y$};

\filldraw
(11,9.5) circle [radius=4pt]
(11.7,6) circle [radius=4pt]
(7,8.5) circle [radius=4pt]
(9,7) circle [radius=4pt]
(9,5.8) circle [radius=4pt]
(9,4) circle [radius=4pt];

\filldraw[green!65!black]
(11,9.5) circle [radius=3pt]
(11.7,6) circle [radius=3pt]
(7,8.5) circle [radius=3pt]
(9,7) circle [radius=3pt]
(9,5.8) circle [radius=3pt]
(9,4) circle [radius=3pt];
\filldraw[black]
(4,3) circle [radius=3pt]
(14,3) circle [radius=3pt]

(11.5,1.8) circle [radius=3pt]
(12.5,1) circle [radius=3pt]
(10.5,1) circle [radius=3pt]
(13.3,0) circle [radius=3pt]
(12,0) circle [radius=3pt]
(11,0)circle [radius=3pt]
(9.7,0)circle [radius=3pt]

(1.5,1.8) circle [radius=3pt]
(2.5,1) circle [radius=3pt]
(.5,1) circle [radius=3pt]
(3.3,0) circle [radius=3pt]
(2,0) circle [radius=3pt]
(1,0)circle [radius=3pt]
(-.3,0)circle [radius=3pt]
;
\filldraw[black]
(16.5,1.8) circle [radius=3pt]
(15.5,1) circle [radius=3pt]
(17.5,1) circle [radius=3pt]
(18,0) circle [radius=3pt]
(17,0) circle [radius=3pt]
(16,0)circle [radius=3pt]
(14.7,0)circle [radius=3pt]

(6.5,1.8) circle [radius=3pt]
(5.5,1) circle [radius=3pt]
(7.5,1) circle [radius=3pt]
(8.3,0) circle [radius=3pt]
(7,0) circle [radius=3pt]
(6,0)circle [radius=3pt]
(4.7,0)circle [radius=3pt]
;
\node(ell8) at (6.6,8.5) {$w$};
\node(ell8) at (11.4,9.5) {$x$};

\node(ell7) at (8.2,10.2) {\color{green!65!black}  \footnotesize $[0,2,4,6,8,$};
\node(ell7) at (9.5,10.21) {\color{red!85!black}  \footnotesize $10$};
\draw[red, thick] (9.29,10.02) rectangle (9.75,10.4);
\node(ell7) at (10.7,10.2) {\color{green!65!black}  \footnotesize $,13,15,17,$};
\node(ell7) at (11.9,10.22) {\color{red!85!black}  \footnotesize $19$};
\draw[red, thick] (11.7,10.02) rectangle (12.15,10.4);
\node(ell7) at (14.8,10.2) {\color{green!65!black}  \footnotesize $,27,29,31,33,35,37,39,41,43]$};

\node(ell8) at (12.2,6) {$y$};
\node(ell7) at (11.9,5.25) {\color{green!65!black}  \footnotesize $[10]$};

\node(ell4) at (9,7.35) {$v$};
\node(ell5) at (6.9,6.8) {\color{green!65!black}  \footnotesize $[10,12,20,22,24,26]$};

\node(ell4) at (9.3,5.5) {$u$};
\node(ell0) at (7.5,5.8) {\color{green!65!black}  \footnotesize $[$};
\draw[red, thick] (7.6,5.6) rectangle (8,6);
\node(ell0) at (7.8,5.8) {\color{red!85!black}  \footnotesize $9$};
\node(ell0) at (8.4,5.8) {\color{green!65!black}  \footnotesize $,11]$};

\node(ell4) at (8.6,4.2) {$z_i$};

\node(ell0) at (10,4.2) {\color{green!65!black}  \footnotesize $[0,2,4,6,$};
\node(ell0) at (11.07,4.2) {\color{red!85!black}  \footnotesize $8$};
\draw[red, thick] (10.9,4) rectangle (11.2,4.4);
\node(ell0) at (13.1,4.2) {\color{green!65!black}  \footnotesize $,12,14,16,21,23,25]$};
 
\node(ell3) at (14.7,3) {$\alpha_2$};
\node(ell0) at (8.4,2.9) {\color{green!65!black}  \footnotesize $[16,18,$};
\node(ell5) at (9.3,2.9) {\color{red!85!black}  \footnotesize $20$};
\node(ell0) at (11.45,2.9) {\color{green!65!black}  \footnotesize $,26,28,30,32,34,36]$};
\draw[red, thick] (9.05,2.7) rectangle (9.55,3.1);

\node(ell3) at (3.4,3.05) {$\alpha_1$};
\node(ell0) at (5,2.9) {\color{green!65!black}  \footnotesize $[14]$};

\node(ell2) at (17,1.8) {$\beta_4$};
\node(ell2) at (15,1.8) {\color{green!65!black} \small $[36,38,40]$};
\node(ell1) at (14.95,1) {{\color{green!65!black}\small  $[38]$}};
\node(ell1) at (16.6,1) {\color{green!65!black} \small $[40,42]$};
\node(ell0) at (14.7,-.4) {\color{green!65!black}  \small$[37]$};
\node(ell0) at (16.05,-.4) {\color{green!65!black} \small $[39]$};
\node(ell0) at (17,-.4) {\color{green!65!black} \small $[41]$};
\node(ell0) at (18,-.4) {\color{green!65!black}  \small $[42]$};

\node(ell2) at (10.9,1.8) {$\beta_3$};
\node(ell2) at (12.2,1.75) {\color{green!65!black} \small $[31]$};
\node(ell1) at (9.6,1) {{\color{green!65!black}\small  $[28,30]$}};
\node(ell1) at (11.6,1) {\color{green!65!black} \small $[32,34]$};
\node(ell0) at (9.7,-.4) {\color{green!65!black}  \small$[27]$};
\node(ell0) at (11.05,-.4) {\color{green!65!black} \small $[29]$};
\node(ell0) at (12.05,-.4) {\color{green!65!black} \small $[33]$};
\node(ell0) at (13.3,-.4) {\color{green!65!black}  \small $[35]$};

\node(ell2) at (5.9,1.8) {$\beta_2$};
\node(ell2) at (8,1.8) {\color{green!65!black} \small $[15,17,19,21]$};
\node(ell1) at (4.95,1) {{\color{green!65!black}\small  $[19]$}};
\node(ell1) at (6.6,1) {\color{green!65!black} \small $[22,24]$};
\node(ell0) at (4.7,-.4) {\color{green!65!black}  \small$[18]$};
\node(ell0) at (6.05,-.4) {\color{green!65!black} \small $[20]$};
\node(ell0) at (7.05,-.4) {\color{green!65!black} \small $[23]$};
\node(ell0) at (8.3,-.4) {\color{green!65!black}  \small $[25]$};

\node(ell2) at (1,1.85) {$\beta_1$};
\node(ell2) at (3.3,1.75) {\color{green!65!black} \small $[5,7,9,11,13]$};
\node(ell1) at (1.4,1) {{\color{green!65!black}\small  $[1,3,5]$}};
\node(ell1) at (2.9,1) {\color{green!65!black} \small $[8]$};
\node(ell0) at (-.3,-.4) {\color{green!65!black}  \small$[1]$};
\node(ell0) at (1.05,-.4) {\color{green!65!black} \small $[3]$};
\node(ell0) at (2.05,-.4) {\color{green!65!black} \small $[7]$};
\node(ell0) at (3.3,-.4) {\color{green!65!black}  \small $[9]$};

\end{tikzpicture}

\caption{Sub-protocol ${\cal P}_3$ applied to $T$. The green number(s) next to a vertex $q$ indicates the time-step(s) in ${\cal P}_3$ at which $q$ is treated and a red boxed number $t$ next to $q$ indicates that $q$ is red at time-step $t$ and was yellow at time-step $t-1$.
}

\label{fig:proto3}
\end{figure}
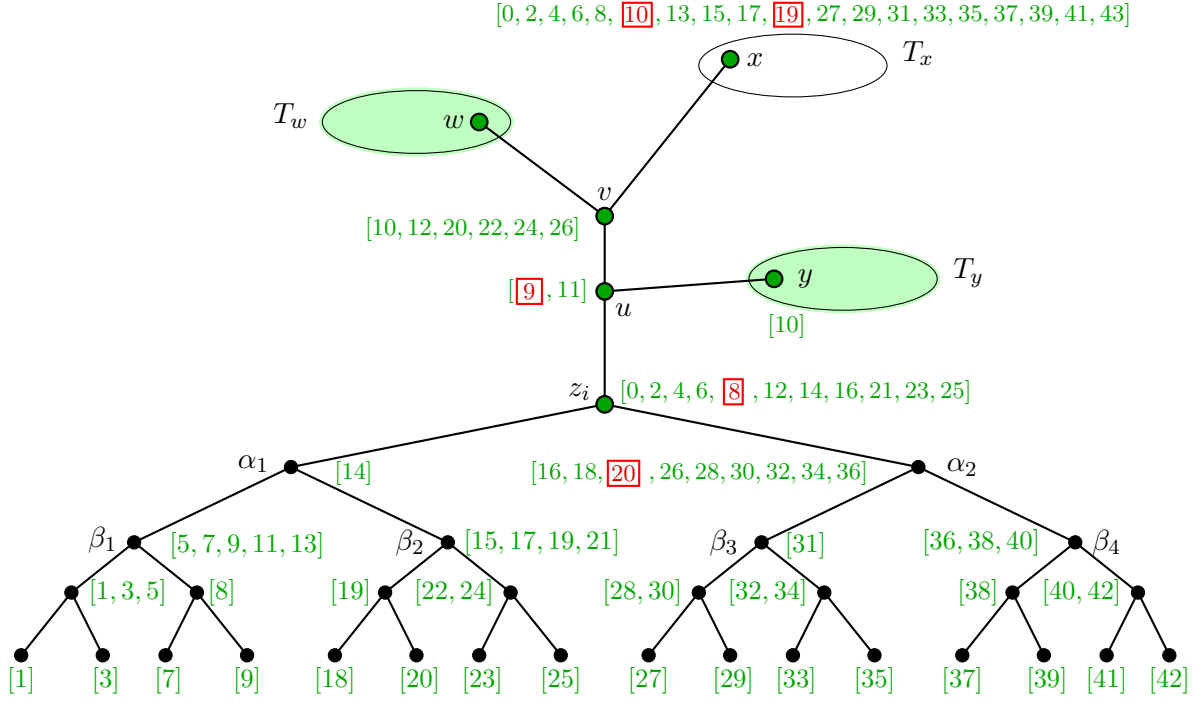

In the proof of Theorem~\ref{thm:bT7}, we refer to Figures~\ref{fig:proto1}, \ref{fig:proto2}, \ref{fig:proto3} as we apply sub-protocols $\mathcal{P}_1, \mathcal{P}_2, \mathcal{P}_3$ respectively.  It is also helpful to refer to Figure~\ref{fig:T7} to see how the sequence of eight sub-protocols fits together to clear $BT(7)$.

\begin{thm}\label{thm:bT7} 
The protocol $\mathcal{P}$, consisting of the sequence of 8 sub-protocols described in Table~\ref{table:8sub}, clears $BT(7)$ and has width $2$.\end{thm}

\begin{proof} Table~\ref{table:8sub} provides a list of the eight stages of protocol $\mathcal{P}$ and the values of $z_i,u,y,v,w$, and $x$ used in each.  After each stage, we keep track of vertices we know are green and ensure that the criteria for applying the appropriate lemma (\ref{lem1}, \ref{lem2}, \ref{lem3}) are met for the next stage. 

In stage 1, we apply $\mathcal{P}_1$ to $T$ shown in Figure~\ref{fig:proto1} 
with $z_i=z_1$, $u = u_1$ $y=z_2$. By Lemma~\ref{lem1}, when this stage is complete, the vertices in tree $T_{z_1}$ are green and vertex $u_1$ is treated at the last time-step.

In stage 2, we apply $\mathcal{P}_1$ to $T$ but this time with $z_i=z_2$, $u = u_1$ $y=z_1$.  By Lemma~\ref{lem1}, when this stage is complete, the vertices in tree $T_{z_2}$ are green and vertex $u_1$ is once again treated at the last time-step.

By Lemma~\ref{lem1}, we also know that $u_1$ is treated every other time-step during stage 2, so the vertices in $T_{z_1}$ remain green at the end of stage 2 (see Figure~\ref{fig:T7}).  Thus, at the end of stage 2, the vertices of $T_{z_1}$, $T_{z_2}$ and vertex $u_1$ are all green.

In stage 3, we apply $\mathcal{P}_2$ to $T$ with $z_i=z_3$, $u = u_2$ $y=z_4$, $v=v_1$, $w=u_1$, $x=r$.  Lemma~\ref{lem2} can be applied because the vertices of $T_w$ are green at the end of stage 2.  By Lemma~\ref{lem2}, when stage 3 is complete, the vertices in $T_{z_1}$, $T_{z_2}$, $T_{z_3}$ are green, along with those in $\{u_1,u_2,v_1,r,z_4\}$ and furthermore, $r$ and $z_4$ are treated at the last time-step.

In stage 4, we apply $\mathcal{P}_3$ to $T$ with $z_i=z_4$, $u=u_2$, $y=z_3$, $v=v_1$, $w=u_1$, $x=r$.  Note that the vertices that are treated at the last time-step in stage 2 (with sub-protocol $\mathcal{P}_2$) correspond to the vertices listed in sub-protocol $\mathcal{P}_3$ that are treated at time-step $0$.

Lemma~\ref{lem3} can be applied because the vertices of $T_w$, $T_y$ and the vertices $\{u,v,x,z_4\}$ are green at the end of stage 3 and vertices $r$ and $z_4$ were treated at the last time-step of stage 3.  Therefore, at the end of stage 4, the vertices in $T_{z_1}$, $T_{z_2}$, $T_{z_3}$, $T_{z_4}$ are green, along with those in $\{u_1,u_2,v_1,r\}$; and $r$ is treated at the last time-step of stage 4.

In stage 5, we apply $\mathcal{P}_2$ to $T$ with $z_i=z_5$, $u=u_3$, $v=v_2$, $y=z_6$, $w=r$, $x=u_4$.  Lemma~\ref{lem2} can be applied because $T_w$ is green at the end of stage 4. Thus at the end of stage 5, by Lemma~\ref{lem2}, the vertices in $T_{z_1}$, $T_{z_2}$, $T_{z_3}$, $T_{z_4}$, $T_{z_5}$ are green, along with those in $\{u_1,u_2,u_3,u_4,v_1,v_2,z_6\}$.  Furthermore, $u_4$ and $z_6$ are treated at the last time-step.

In stage 6, we apply $\mathcal{P}_3$ to $T$ with $z_i=z_6$, $u=u_3$, $y=z_5$, $v=v_2$, $w=r$, $x=u_4$.  Lemma~\ref{lem3} can be applied because $z_6$ and $u_4$ are treated at the last time-step in stage 5 and the vertices in $T_w$, $T_{z_5}$ are green, along with those in $\{z_6,u_3,v_2,u_4\}$.  By Lemma~\ref{lem3}, at the end of stage 6, the vertices in $T_{z_1}$, $T_{z_2}$, $T_{z_3}$, $T_{z_4}$, $T_{z_5},T_{z_6}$ are green, along with $\{u_1,u_2,u_3,u_4,v_1,v_2\}$.

In stage 7, we apply $\mathcal{P}_1$ to $T$ with $z_i=z_7$, $u=u_4$, $y=z_8$. By Lemma~\ref{lem1}, $u_4$ is treated every second time-step (including the last) and at the end of stage 7, the vertices in $T_{z_7}$ are green, along with all vertices that were green at the end of stage 6.

In stage 8, we apply $\mathcal{P}_1$ to $T$ with $z_i=z_8$, $u=u_4$, $y=z_7$.  By Lemma~\ref{lem1}, $u_4$ is treated every second time-step and at the end of stage 8, the vertices in $T_{v_8}$ are green, along with all the vertices that were green at the end of stage 7.

Since each sub-protocol has width 2, the full protocol $\mathcal{P}$ has width 2.
\end{proof}

The protocol given in the proof of Theorem~\ref{thm:bT7} contains vertices that become red after they have been treated. A protocol that clears a graph is {\it monotone} if for each vertex, once it is treated, it never becomes red again. Thus, our protocol for $BT(7)$ is non-monotone.

\section{Small vertex cuts resulting in many components}\label{sec:K4s}

In this section, we use path decompositions (that are not necessarily optimal)  to bound the treatment number of a graph.  We  construct a family of graphs, each with a cut-vertex, and show that the treatment number of graphs in the family depends on the number of components when the cut-vertex is removed. Additionally, in contrast to our use of Corollary~\ref{lower-bnd} in the proof of Theorem~\ref{thm:depth8}, we construct a related family of graphs that satisfies the $2$-bound but has arbitrarily large treatment number.

The proof of Theorem~\ref{pathwidth-bnd-thm}  gives a method for using path decompositions of a graph to construct protocols  that have an upper bound on their width and   also clear the graph. One way to construct a  path decomposition of a graph $H$ is to find a vertex cut $W$ of $H$, where the  components of $H-W$ are $C_1, C_2, \ldots, C_{\ell}$,  and use the path decomposition whose consecutive bags  are $W\cup C_1,  W\cup C_2, \ldots, W\cup C_{\ell}$.  We record this bound in the next proposition.

\begin{prop} \label{prop:upperbd-all}
If $W$ is a vertex cut of graph $H$   and  the components of $H-W$ are $C_1, C_2, \ldots, C_{\ell}$, then $\tau(H)\leq \left \lceil \frac{|W|}{2}\right \rceil+\max \{\tau(C_j): 1 \le j \le \ell\}$.  
\end{prop} 

\begin{proof}

Let $P_j$ be a protocol that clears $C_j$ and has width $\tau(C_j)$ for $1\leq j\leq \ell$. 
The following protocol clears $H$ and has the desired width:  follow $P_1, P_2, \ldots, P_{\ell}$ consecutively and at each time-step, in addition to the vertices treated in  these protocols, also 
 treat alternate halves of the vertices in $W$.  
\end{proof}

In the special case when $|W|=1$ and the number of vertices treated  per time-step  is less than the maximum at every other time-step in the protocols, we can reduce the upper bound by 1. 

\begin{cor} \label{prop:upperbd-odd}
Let $H$ be a graph with a cut-vertex $w$, and let the components of $H-w$ be $C_1, C_2, \ldots, C_{\ell}$.  If  $k=\max\{\tau(C_j): 1 \le j \le \ell\}$ and   for $1\leq j\leq \ell$ component $C_j$ has a protocol $P_j$ that  treats $k-1$ or fewer vertices  per time-step at alternate time-steps, then $\tau(H)\leq k$.  
\end{cor} 

\begin{proof}
Similar to the proof of Proposition~\ref{prop:upperbd-all}, let $P_j$ be a protocol that clears $C_j$ and has width $\tau(C_j)$ for $1\leq j\leq \ell$. We clear $H$ by following $P_1, P_2, \ldots, P_{\ell}$ consecutively, and  treating $w$ on alternate time-steps when the protocol $P_j$ treats $k-1$ or fewer vertices per time-step. If $k$ vertices are treated in the last time-step of $P_j$ and the first of $P_{j+1}$ for any $j$, then add an intermediate time-step in which only $w$ is treated.
\end{proof}

In the next definition we introduce the family of  graphs $H_{m,n}$.  Each graph has a central cut-vertex $r$ and when $r$ is deleted, the resulting graph consists of $m$ copies of the complete graph $K_n$.  The advantage  of using complete graphs in the construction is that the presence of any red vertex in a complete graph immediately changes the color of all non-treated green vertices to yellow.

\begin{defn}\label{Hmn} \rm The graph  $H_{m,n}$ consists of a root $r$ and $m$ copies of a complete graph $K_n$ where  exactly one vertex in each copy is adjacent to $r$. We label the vertices in the copies of $K_n$ as $V^1, V^2, \ldots, V^m$ and  for each $i$,  we let vertex $x_i$ be the vertex in $V^i$ that is adjacent to the root $r$. 

\end{defn}

The graphs $H_{m,n}$ are a generalization of $K_{m,1}$ where a complete graph replaces each leaf. Since $K_{m,1}$ is a caterpillar,  we know $\tau(K_{m,1})=1$ for any value of $m$ by \cite[Theorem 4.3]{FirstTreatPaper}. Thus, we might expect that the size of $m$ would not affect the treatment number. If the clique size $n$ is odd, $\tau(H_{m,n})$ remains constant when $n$ is fixed and $m$ is varied. In contrast, we prove that when $n$ is even, there is a threshold for $m$: when $m$ meets (or exceeds) the threshold, the treatment number of the graph changes. Thus, for fixed $n=2q$, the treatment number of the graph $H_{m,2q}$ is determined by $m$ and not by the pathwidth of $H_{m,2q}$. 

Before the main theorem, we demonstrate a protocol for $H_{6,4}$ by listing the treatment sets in Table~\ref{table:6K4s}. For $1\leq i\leq 6$, we label $V^i = \{x_i,y_i,z_i,w_i: 1 \le i \le 6\}$ where vertices $x_i,y_i, z_i, w_i$ induce a clique. The action of the protocol is described in Table~\ref{Table:6K4s2} with 22 time-steps. The blank entries in Table~\ref{Table:6K4s2} indicate that a vertex is red. For instance, vertex $r$ is red at time-steps $1,2,5,9$, and $10$.  It is straightforward to verify that the entries in Table~\ref{Table:6K4s2} are correct and therefore our width 2 protocol clears $H_{6,4}$ and $\tau(H_{6,4}) \leq 2$. Since $H_{6,4}$ contains $K_4$, by Proposition~\ref{subgraph}, we know $\tau(H_{6,4})=2$.

\begin{table}[htbp] \small 
\[ 
\setlength\arraycolsep{3.3pt}
\begin{array}{@{}c@{}||c|c|c|c|c|c|c|c|c|c|c|c|c|c|c|c|c|c|c|c|c|c}
t &{\bf 1} & {\bf 2} & 3 & {\bf 4} & {\bf 5} & 6 & 7 & {\bf 8} & {\bf 9} & 10 &11 & 12 & 13 & {\bf 14} & {\bf 15} & 16 & 17 & {\bf 18}& {\bf 19} & 20 & {\bf 21} & {\bf 22}\\ \hline
\  & z_1 & x_1 & r & z_2 & x_2 & r & r &         z_3 & x_3 & x_1 & r &r & x_5 & x_4 & z_4 & r &    x_5 & z_5 & x_5 & r & z_6 & x_6\\
A_t~ & w_1 & y_1 & x_2 & w_2 & y_2  & x_1 & x_3 & w_3 & y_3 & x_2 & x_3 &x_4& x_6 & y_4 & w_4 &  & x_6 & w_5 & y_5 & x_6 & w_6 & y_6\\
\end{array}
\]
\caption{A  width $2$ protocol that clears the graph $H_{6,4}$. The columns shown in bold indicate consecutive time-steps in which vertices of the same $K_4$ clique are treated.}

\label{table:6K4s}
\end{table}

Our protocol for $H_{6,4}$  has the property that for each $i: 1 \le i \le 6$, there are two consecutive time-steps in which all vertices treated are in $V^i$, and only in $V^i$. We will show in Lemma~\ref{two-conseq-lem}  that this is always the case for a width $q$ protocol that clears $H_{m,2q}$ where $q \ge 2$ and $m\geq 2q+3$. 
Observe in Table~\ref{Table:6K4s2} that vertices $r$, $x_5$, and $x_6$ turn red after they have been treated, for example, $r$ turns red at time-step $5$ after being treated at time-step $3$. Thus, our protocol for $H_{6,4}$ is not monotone, and this will also be the case for our protocols for $H_{m,2q}$ in the next theorem.

\begin{table}[htb] \small 
\[
\setlength\arraycolsep{3.5pt}
\begin{array}{@{}c@{}||c|c|c|c|c|c|c|c|c|c|c|c|c|c|c|c|c|c|c|c|c|c|c}
~ & 1 & 2 & 3 & 4 & 5 & 6 & 7 & 8 & 9 & 10 & 11 & 12 & 13 & 14 & 15 & 16 & 17 & 18 & 19 & 20 & 21 & 22\\
\hline 
r & ~ & ~ & \g & \y & ~ & \g & \g & \y & ~ & ~ & \g & \g & \g & \g & \y & \g & \g & \g & \y & \g & \g & \g \\
w_1, z_1~ & \g & \g & \g & \g & \g & \g & \g & \g & \g & \g & \g & \g & \g & \g & \g & \g & \g & \g & \g & \g & \g & \g \\
x_1 & ~ & \g & \g & \g & \y & \g & \g & \g & \y & \g & \g & \g & \g & \g & \g & \g & \g & \g & \g & \g & \g & \g\\
y_1 & ~ & \g & \g & \g & \g & \g & \g & \g & \g & \g & \g & \g & \g & \g & \g & \g & \g & \g & \g & \g & \g & \g \\

w_2, z_2 & ~ & ~ & ~ & \g & \g & \g & \g & \g & \g & \g & \g & \g & \g & \g & \g & \g & \g & \g & \g & \g & \g & \g\\
x_2 & ~ & ~ & \g & \y & \g & \g & \g & \g & \y & \g & \g & \g & \g & \g & \g & \g & \g & \g & \g & \g & \g & \g\\
y_2 & ~ & ~  &~ & ~ & \g& \g &\g &\g & \g& \g& \g & \g & \g & \g & \g & \g & \g & \g & \g & \g & \g & \g\\

w_3, z_3 & ~ & ~ & ~ & ~ & ~ & ~ & ~ & \g & \g & \g & \g & \g & \g & \g & \g & \g & \g & \g & \g & \g & \g & \g\\
x_3 & ~ & ~ & ~ & ~ & ~ & ~ & \g & \y &\g & \y & \g & \g & \g & \g & \g & \g & \g & \g & \g & \g & \g & \g\\
y_3 & ~ & ~  &~ &~&~&~&~&~&\g& \g & \g & \g & \g & \g & \g & \g & \g & \g & \g & \g & \g & \g\\

w_4, z_4 & ~ & ~ & ~ & ~ & ~ & ~ & ~ & ~ & ~ & ~ & ~ & ~ &~ & ~ & \g &  \g& \g& \g& \g& \g& \g& \g\\
x_4 & ~ & ~ & ~ & ~ & ~ & ~ & ~ & ~ & ~ & ~ & ~  & \g & \y & \g & \g & \g& \g& \g& \g& \g& \g& \g \\
y_4 &~ & ~ & ~ & ~ & ~ & ~ & ~ & ~ & ~ & ~ & ~ & ~ &~ & \g & \g &  \g& \g& \g& \g& \g& \g& \g\\

w_5, z_5 & ~ & ~ & ~ & ~ & ~ & ~ & ~ & ~ & ~ & ~ & ~ & ~  & ~& ~& ~ &~ & ~& \g& \g& \g& \g& \g \\
x_5 & ~ & ~ & ~ & ~ & ~ & ~ & ~ & ~ & ~ & ~ & ~ & ~ & \g & \y & ~ & ~&\g & \y& \g& \g& \g& \g\\
y_5 &~ & ~ & ~ & ~ & ~ & ~ & ~ & ~ & ~ & ~ & ~ & ~ &~ &~& ~ &~&~&~& \g& \g& \g& \g\\

w_6, z_6 & ~ & ~ & ~ & ~ & ~ & ~ & ~ & ~ & ~ & ~ & ~ & ~ &~ &~&~&~&~&~&~&~& \g & \g\\
x_6 & ~ & ~ & ~ & ~ & ~ & ~ & ~ & ~ & ~ & ~ & ~ & ~ & \g & \y &~&~& \g& \y &~&\g & \y & \g \\
y_6 & ~ & ~ & ~ & ~ & ~ & ~ & ~ & ~ & ~ & ~ & ~& ~ & ~ & ~&~& ~&~&~& ~&~&~& \g\\
\hline
|G_t| & 2 & 4 & 6 & 6 & 7 & 9 & 10 & 10 & 10 & 11 & 13 & 14 & 15 & 15 & 16 & 17 & 19 & 19 & 20 & 22 & 23 & 25 \\
\end{array}
\]

\caption{The entries in this table show which vertices are green (\g) and yellow (\y) at each time-step for the protocol  given in Table~\ref{table:6K4s}.  Vertices $w_i,z_i$ are grouped together as they are always treated together, and blank entries indicate that a vertex is red. }

 \label{Table:6K4s2}
\end{table}

We begin with an upper bound for $\tau(H_{m,n})$ and the computation of $\tau(H_{m,n})$ for $n$ odd. 

\begin{prop} \label{prop:upperbdK}
For any positive integers $m,n$, we have $\tau(H_{m,n})\leq \left \lceil \frac{n+1}{2}\right \rceil $. If $n$ is odd, then $\tau(H_{m,n})=\frac{n+1}{2} $.
\end{prop} 
\begin{proof} Since the root $r$ is a cut-vertex in $H_{m,n}$ and $\tau(K_{n})=\lceil \frac{n}{2}\rceil$, by Proposition~\ref{prop:upperbd-all}, $\tau( H_{m,n})\leq 1+\lceil \frac{n}{2}\rceil$. If $n$ is even, then $1+\lceil \frac{n}{2}\rceil=\left \lceil \frac{n+1}{2}\right \rceil $. 

If $n$ is odd, then $ \lceil \frac{n}{2}\rceil = \left \lceil \frac{n+1}{2}\right \rceil $, 
and we can choose to treat $\frac{n-1}{2}$ vertices in alternate time-steps, with $\frac{n+1}{2}$ in the other time-steps. By Corollary~\ref{prop:upperbd-odd},  we know $\tau(H_{m,n})\leq   \frac{n+1}{2}$. Since $H_{m,n}$ contains $K_n$, by Proposition~\ref{subgraph}, $\tau( H_{m,n})\geq \lceil  \frac{n}{2}\rceil$, hence $\tau(H_{m,n})=\frac{n+1}{2}$. 
\end{proof}

Next we prove two lemmas that will be used in the computation of $\tau(H_{m,n})$ where $n$ is even. We focus on the case where $m=2q+3$. If $m<2q+3$, a protocol similar to that described above for $H_{6,4}$ can be constructed to show that the treatment number is $q$. The main theorem of this section will prove that $\tau(H_{m,2q})=q+1$ if $m\geq 2q+3$.

\begin{lemma}

Let $(A_1, A_2, \ldots, A_N)$ be  a protocol  of width $q$ for the graph $H_{2q+3,2q}$, where $q \ge 2$. 
    At any time-step $t$, for any $i$,  either $|V^i\cap G_t|\leq q$, or  $|V^i\cap G_t|=2q$, or  $|V^i\cap G_t|=2q-1$, and in the last case,    $x_i$ is yellow.  
    \label{num-green-lem}
\end{lemma}

\begin{proof}
Let  $|V^i\cap G_t|=k$, and first suppose that $k\leq 2q-2$. Then $V^i$ contains a non-green vertex $y_i$ at time-step $t$ with $y_i \neq x_i$. If $y_i$ is red at time-step $t$, then $V^i\cap G_t=V^i\cap A_t$, and thus 
$|V^i\cap G_t| \le |A_t| \le q$, so $k \le q$.
 If $y_i$ is yellow at time-step $t$, then there is a red vertex $z_i\in V^i$, because $y_i$ is not adjacent to the root $r$. Then  at time-step $t$,  all other vertices  in $V^i - A_i$  turn red because they are adjacent to vertex $z_i$, and again, $V^i\cap G_t=V^i\cap A_t$,   so  as before, $k\leq q$. 

Now suppose that $k=2q-1$. By similar reasoning, every vertex in $V^i$ that is not $x_i$ must be green. If $x_i$ is red, then again, $V^i\cap G_t=V^i\cap A_t$, resulting in $k\leq q$, a contradiction. Hence $x_i$ is yellow as desired.
The remaining case, $k=2q$, is our third possibility.
\end{proof}

\begin{lemma}
If $(A_1, A_2, \ldots, A_N)$ is  a  width $q$ protocol  for $H_{2q+3,2q}$, where $q \ge 2$,  and $V^i$ has fewer than $2q-1$ green vertices at time-step $t$ and at least $2q-1$ green vertices at time-step $t+1 $  
for some $i$,  then $A_t \cup A_{t+1} \subseteq V^i$. 
\label{two-conseq-lem}
\end{lemma}

\begin{proof}
By Lemma~\ref{num-green-lem},  at each time-step, the number of green vertices in $V^i$ is in one of the following categories:  (i) less than $q$, (ii) $q$, (iii) $2q-1$, or (iv) $2q$.  By hypothesis, the number of green vertices in $V^i$ is in category (i)  or (ii) at time-step $t$ and in (iii) or (iv) at time-step $t+1$.  Vertices turn green only when they are treated, so the number of green vertices in $V^i$ increases by at most $q$ in any one time-step.  Therefore, it is impossible to go from category (i) to category (iv) in one time-step.
We next show that it is impossible to go from  category (i) to category (iii) in one time-step.
  By Lemma~\ref{num-green-lem}, if there are exactly $2q-1$ green vertices at time-step $t+1$, then $x_i$ is yellow at time-step $t+1$, so $x_i$ is green at time-step $t$.   This means there are at most $q-2$ green vertices in $V^i$ at time-step $t$ that remain green  at time-step $t+1$ without being treated.  So the remaining $q+1$ must be treated, which contradicts the protocol being width $q$.

  By hypothesis, the number of green vertices in $V^i$ at time-step $t+1$ is in  category (iii) or (iv), so the number of green vertices in $V^i$ at time-step $t$ must be in category (ii).  Thus there are $q$ vertices in $V^i$ that are non-green at time-step $t$, and if we do not include all of them in $A_{t+1}$,
  then the set of green vertices  in $V^i$ at time $t+1$ will be $V^i \cap A_{t+1}$, and this will have size $q$, which is less than $2q-1$ a contradiction. Thus $A_{t+1} \subseteq V^i$.

Hence we know $A_{t+1} \subseteq V^i$ and  $|V^i \cap G_t| = q$, and 
it remains to show $A_t \subseteq V^i$. If  $|V^i \cap G_{t-1}| \le q$, all  vertices in $V^i$ will be non-green at time $t$ except those in $A_t$. Thus $|A_t| = q$ and $A_t \subseteq V^i$.  Otherwise,  $|V^i \cap G_{t-1}| \ge 2q-1$.   All vertices in $V^i$, other than $x_i$, have the same neighbor set.  If they have no red neighbor at time-step $t$, then they   remain green at time $t$, a contradiction, since $2q-1 > q$.  Hence these vertices must have a red neighbor at time-step $t$ and then the only vertices of $V^i$ that are green at time-step $t$ are those in $A_t$, so $|A_t| = q$ and $A_t \subseteq V^i$. 
\end{proof}

We are now ready to prove the main result of this section.   

\begin{thm}\label{thm:root-complete-g} 
 If $q \ge 2$ and $m \geq 2q+3$, then $\tau(H_{m,2q}) = q+1$.

\end{thm}

\begin{proof} Using Proposition~\ref{subgraph}, since $K_{2q}$ is a subgraph of $H_{m,2q}$, we have $\tau(H_{m,2q})\geq q$.   By Proposition~\ref{prop:upperbdK}, we know $\tau(H_{m,2q}) \le q+1$, so $\tau(H_{m,2q}) \in \{q,q+1\}$.  If $m\geq 2q+3$ then as $H_{2q+3,2q}$ is a subgraph of $H_{m,2q}$, we have $q+1 \ge \tau(H_{m,2q})\geq  \tau(H_{2q+3,2q})$. Thus, it suffices to show that $q+1 \le \tau(H_{2q+3,2q})$.

For a contradiction, assume 
 $\tau(H_{2q+3,2q}) \le q$.  Let $(A_1, A_2, \ldots, A_N)$ be a width $q$ protocol that clears 
$H_{2q+3,2q}$.
We  consider the first time-step $t+2$ after which there are always at least  $q+2$  of the $V^i$ that each have at least $2q-1$ green vertices. Without loss of generality, we may assume that these  cliques are $V^1, V^2, \ldots, V^{q+2}$, and that $V^{q+2}$ is the latest of these  to have at least $2q-1$ green vertices. By Lemma~\ref{two-conseq-lem},   $A_{t+1} \cup A_{t+2} \subseteq V^{q+2}$, thus no vertices outside of  $V^{q+2}$ are treated at time-steps $t+1$ and $t+2$.

If there exists $i: 1\leq i\leq q+1$,  and a vertex of $V^i$ that is non-green at time-step $t$ or $t+1$, then all vertices of $V^i$ will be red at time-step $t+2$, a contradiction.  Thus, all vertices in 
$V^1 \cup V^2 \cup \cdots \cup V^{q+1}$ are green at time-steps $t$ and $t+1$.  
Since $x_1, x_2, \ldots x_{q+1}$ are neighbors of root $r$, then $r$ must be green at time-step $t$. 

First suppose $r$ is red at time-step $t+2$.  Then vertices   $x_1, x_2, \ldots, x_{q+1}$ become yellow at time-step $t+2$, and since $|A_{t+3}|\leq q$, at least one of $x_1, x_2, \ldots, x_{q+1}$ is not in $A_{t+3}$ and becomes red at time-step $t+3$, say it is $x_1$. 
Then all vertices in $V^1$ are non-green at time-step $t+3$ and there are at most $q+1$ cliques with at least $2q-1$ green vertices. This contradicts our choice of $t+2$.  

Otherwise,   $r$ is  green or yellow at time-step $t+2$. Since $r \not\in A_{t+1}$, none of the neighbors $x_{q+3}, x_{q+4}, \ldots, x_{2q+3}$ can be red at 
  time-step $t+1$. 
Since $|A_t|\leq q$, there exists $j: q+3 \le j \le 2q+3$ for which $V^j \cap A_t = \emptyset$, and therefore, no vertices of $V^j$ are treated at  any of time-steps 
  $t, \, t+1, \, t+2$. Since $x_j$ is not red at time-step $t+1$ and not treated at time-steps $t$ or $t+1$, then $x_j$ is treated at time-step $t-1$ or earlier and remains green at time-step $t$ without being treated. Thus, no vertex in $V^j$ is   red at time-step $t$. If  there exists $y_j \in V^j$  that is yellow at time-step $t$ then $y_j \neq x_j$ and $y_j$ will only turn red due to adjacency with a third vertex $z_j \in V^j$ that is red at time-step $t$, a contradiction.
  Thus, every vertex in $V^j$ is green at time-step $t$,   which means  that by time-step $t$, there are already $q+2$ of the $V^i$ that each have at least $2q-1$ green vertices, contradicting our choice of $t+2$. 
\end{proof}

The protocol in Table~\ref{table:6K4s}, which  clears $H_{6,4}$ with  at most two vertices treated per time-step, uses two consecutive time-steps to clear the vertices in each $V^i$, and these time-steps are necessary by Lemma~\ref{two-conseq-lem}. When clearing $V^2$ and $V^3$, the root $r$ is allowed to become red again after having been treated. However, if $r$ turns red after the vertices in three of the $V^i$  have become green,  then at least three  green neighbors of $r$ turn yellow, and  at most two of them can be treated at the same time-step, so  at the next time-step, the vertices in at least one  of these $V^i$  will no longer be green. Thus, $r$ must remain green or yellow in the protocol when clearing the vertices in the last three $V^i$.  
The vertices in the fourth and fifth $V^i$ are cleared in the necessary two consecutive time-steps by   first treating the neighbors of $r$ in the remaining red $V^i$'s, before proceeding. A protocol analogous to that in Table~\ref{table:6K4s} shows that if $m\leq 2q+2$, then $H_{m,2q}$ can be cleared with at most $q$ vertices treated per time-step. We record this in the next observation. 

\begin{obs} \label{obs:2q-2apples} \rm If $1\leq m\leq 2q+2$, then $\tau(H_{m,2q})=q$. \end{obs}

\noindent Hence the treatment number of $H_{m,2q}$ is $q$ if $1\leq m\leq 2q+2$ and $q+1$ if $m\geq 2q+3$. 

In our proof  of Theorem~\ref{thm:depth8}, we establish the lower bound   $\tau(BT(8)) \ge 3$  by showing that $BT(8)$ does not satisfy the $2$-bound, and applying Corollary~\ref{lower-bnd}, but this technique has its limitations.  For our final result, we build onto $H_{m,2q}$ to construct the graph $H'_{m,2q}$ that also has treatment number equal to $q+1$. Thus graphs in this family have arbitrarily large treatment number.   We show that $H'_{m,2q}$  satisfies the $2$-bound and therefore,  by Definition~\ref{gamma-bound-def},
satisfies the $\gamma$-bound for any $\gamma \ge 2$.  Thus our technique of using the contrapositive of Corollary~\ref{lower-bnd}
to prove a lower bound on the treatment number will not be helpful for this family of graphs.

\begin{defn} \rm Let $q \geq 2$.  The graph $H'_{m,2q}$ is constructed from the graph $H_{m,2q}$ by adding an additional path $P$ attached to root $r$ with $2q-1$ new interior vertices.  Label the vertices of $P$ as $v_1, v_2, \ldots, v_{2q}$ where $v_1$ is adjacent to $r$ and $v_i$ is adjacent to $v_{i+1}$ for $1 \le i \le 2q-1$. 
\label{HQ-def}
\end{defn}
We observe  that $|V(H'_{m,2q})| = 2q(m+1)+ 1.$

\begin{thm}
If $m\geq 2q+3$, then the graph $H'_{m,2q}$, constructed in Definition~\ref{HQ-def}, satisfies the $2$-bound and has $\tau(H'_{m,2q}) = q+1$.
\label{HQ-thm}
\end{thm}

\begin{proof}
First we show $\tau(H'_{m,2q}) = q+1.$
In Theorem~\ref{thm:root-complete-g}  we show $\tau(H_{m,2q}) = q+1$, and by Proposition~\ref{subgraph}, since $H_{m,2q}$ is a subgraph of $H'_{m,2q}$ we know $q+1 \le \tau(H'_{m,2q})$.  Indeed, $\tau(H'_{m,2q}) = q+1$ since we can use the protocol from the proof of Proposition~\ref{prop:upperbdK} to clear subgraph $H_{m,2q}$ and then clear the vertices on $P$ in two more steps, the first with treatment set $\{v_1, v_2, \ldots, v_q\}$ and the second with $\{v_{q+1}, v_{q_2}, \ldots, v_{2q} \}$.

Next we construct a set $S$ of vertices in $H'_{m,2q}$ for each cardinality $p: 1\leq p\leq |V(H_q)|-1$  with $|bnd(S)| \le 2$.  Write $p=(2q)k+\ell$ with $0 \le k \le m$  and $0\leq \ell < 2q$.  Since $p \ge 1$, we know $k$ and $\ell$ are not both $0$.  Let $S$ consist of the vertices in $k$ copies of $K_{2q}$, and when $\ell \neq 0$ we also include   $\{v_1, v_2, \ldots, v_{\ell}\}$.
Then $\bnd(S) \subseteq \{r,v_{\ell + 1} \}$, so $|\bnd(S)| \le 2$ and $H'_{m,2q}$ satisfies the $2$-bound.
 \end{proof}

\section{Concluding remarks}

 We conclude with a discussion of the  related concept of  isoperimetric peak and some open questions.
In our work proving that the complete binary tree of depth $8$ has treatment number $3$ we described all sets of vertices having a fixed boundary size of $3$ or smaller.  
A related problem is minimizing $|\bnd(S)|$ over all sets $S$ of vertices with $|S| = k$,  see \cite{H04} for more information.  This latter problem, called the \emph{vertex isoperimetric problem,} fixes $|S|$ while our problem fixes $|\bnd(S)|$. Most research on the vertex isoperimetric problem focuses on determining the vertex isoperimetric peak of a graph $H$, which is the quantity $\max_{1 \leq k \leq |V(H)|} \min_{S \subset V(H), |S|=k} |\bnd(S)|$ and is denoted by $\Phi_v(H)$.  The vertex isoperimetric peak does not have a direct relationship to the treatment number of a graph. 
In \cite{FirstTreatPaper} we show $\Phi_v(P_n \square P_n)=n$ and   $\tau(P_n \square P_n) = \lceil \frac{n+1}{2}\rceil$, so the vertex isoperimetric peak can be larger than  or equal to the treatment number.  Figure~\ref{fig-iso} gives an example of a graph $H$ which has 
  $\tau(H) = 2$ and the table in the figure shows that $\Phi_v(H) = 1$.

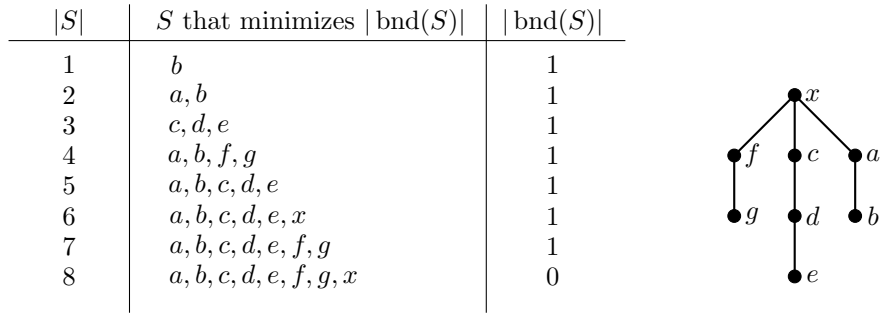
\begin{figure}[htbp] \small 
\centering 
\begin{tikzpicture}[scale=.8]
\draw (-1,3.9)--(9.2,3.9); 
\draw (1,-.6)--(1,4.5); 
\draw (6.9,-.6)--(6.9,4.5); 

\draw[thick] (12,0)--(12,3);
\draw[thick] (11,1)--(11,2)--(12,3)--(13,2)--(13,1);

\filldraw[black]

(12,3) circle [radius=3pt]
(12,2) circle [radius=3pt]
(12,1)circle [radius=3pt]
(12,0)circle [radius=3pt]
(11,2) circle [radius=3pt]
(13,2) circle [radius=3pt]
(11,1) circle [radius=3pt]
(13,1)circle [radius=3pt]
;
\node(0) at (0,4.2)  {\small $|S|$};
\node(0) at (4,4.2)  {\small $S$ that minimizes $|\bnd(S)|$};
\node(0) at (8,4.2)  {\small $|\bnd(S)|$};
\node(0) at (0,3.5)  {$1$};
\node(0) at (0,3)  {$2$};
\node(0) at (0,2.5)  {$3$};
\node(0) at (0,2)  {$4$};
\node(0) at (0,1.5)  {$5$};
\node(0) at (0,1)  {$6$};
\node(0) at (0,.5)  {$7$};
\node(0) at (0,0)  {$8$};

\node(1) at (12.3,3) {$x$};
\node(1) at (12.3,2) {$c$};
\node(1) at (12.3,1) {$d$};
\node(1) at (12.3,0) {$e$};
\node(1) at (13.3,2) {$a$};
\node(1) at (13.3,1) {$b$};
\node(1) at (11.3,2) {$f$};
\node(1) at (11.3,1) {$g$};

\node(2) at (1.78,3.5)  {$b$};
\node(2) at (1.97,3)  {$a,b$};
\node(2) at (2.15,2.5)  {$c,d,e$};
\node(2) at (2.37,2)  {$a,b,f,g$};
\node(2) at (2.56,1.5)  {$a,b,c,d,e$};
\node(2) at (2.8,1)  {$a,b,c,d,e,x$};
\node(2) at (3,.5)  {$a,b,c,d,e,f,g$};
\node(2) at (3.22,0)  {$a,b,c,d,e,f,g,x$};

\node(3) at (8,3.5)  {$1$};
\node(3) at (8,3)  {$1$};
\node(3) at (8,2.5)  {$1$};
\node(3) at (8,2)  {$1$};
\node(3) at (8,1.5)  {$1$};
\node(3) at (8,1)  {$1$};
\node(3) at (8,.5)  {$1$};
\node(3) at (8,0)  {$0$};
\end{tikzpicture}

\caption{A graph $H$ with $\tau(H)=2$ and $\Phi_v(H)=1$.} 

\label{fig-iso}\end{figure}

In this paper we have studied the minimum width of a protocol that clears a graph without regard for the number of time-steps involved.  

\begin{ques} What is the minimum number of time-steps required to clear a graph $H$  using a protocol of width $\tau(H)$?  What about functions of the minimum number of time-steps and width? 
\end{ques}

\begin{ques}
What are other interesting families of graphs with a jump in the treatment number depending on the number of components remaining  after a vertex cut?
\end{ques}

\begin{ques}
What is $\tau(BT(d))$ for $d\geq 11$?  
\end{ques}

\noindent {\bf Acknowledgement}: M.E. Messinger acknowledges research support from the Natural Sciences and Engineering Research Council of Canada (grant application 2018-04059).

 \end{document}